\documentclass[twoside]{article}

\usepackage[margin=1in]{geometry}

\title{On Bayesian Consistency for Flows Observed Through a Passive Scalar}
\author{Jeff Borggaard, Nathan E. Glatt-Holtz, Justin A. Krometis\\
  \scriptsize{emails: jborggaard@vt.edu, negh@tulane.edu, jkrometis@vt.edu}}
\date{}

\usepackage{amsfonts, amssymb, amsmath, amsthm, multicol,bbm}
\usepackage[colorlinks=true, pdfstartview=FitV, linkcolor=blue,
            citecolor=blue, urlcolor=blue]{hyperref}
\usepackage[usenames]{color}
\definecolor{Red}{rgb}{0.7,0,0.1}
\definecolor{Green}{rgb}{0,0.7,0}
\usepackage{accents}
\usepackage{comment,url}

\numberwithin{equation}{section}

\usepackage{enumitem}
\usepackage[colorinlistoftodos]{todonotes}
\usepackage[capitalize,nameinlink,noabbrev]{cleveref} 
\usepackage[section]{algorithm}      
\usepackage{algpseudocode}    
\usepackage{wrapfig}                            
\usepackage{tabularx}
\usepackage{bigints}

\usepackage{tikz}
\usetikzlibrary{shapes,arrows,positioning}

\usepackage{mathtools} 
\usepackage{xparse}
\usepackage{mathrsfs}

\newcommand{\R}{\mathbb{R}}           
\newcommand{\Z}{\mathbb{Z}}           

\newcommand{\Ito}{It\^{o}~}

\DeclarePairedDelimiterX\innerp[2]{\langle}{\rangle}{#1,#2}
\newcommand{\ip}[2]{ \innerp{#1}{#2} }

\newcommand{\norm}[1]{\left\lVert#1\right\rVert}

\NewDocumentCommand\Exp{g}{%
    \mathbb{E}\IfNoValueTF{#1}{}{\left[ #1 \right]}%
}

\NewDocumentCommand\Prob{g}{%
    \mathbb{P}\IfNoValueTF{#1}{}{\left[ #1 \right]}%
}

\newcommand{\x}{\mathbf{x}}

\newcommand{\kbf}{\mathbf{k}}

\newcommand{\spatdom}{\mathbb{T}^2}

\newcommand{\conductivity}{\kappa}
\newcommand{\pdesol}{\theta}

\newcommand{\vfieldnobf}{v}
\newcommand{\vfield}{\mathbf{\vfieldnobf}}
\newcommand{\vk}{\vfieldnobf_{\kbf}}
\newcommand{\vtrue}{\vfield^{\star}}
\newcommand{\vfspace}{H}

\newcommand{\bfU}{\mathbf{u}}
\newcommand{\Xmet}{\rho}

\newcommand{\G}{\mathcal{G}}
\newcommand{\Obs}{\mathcal{O}}
\newcommand{\Sol}{\mathcal{S}}

\newcommand{\data}{\mathcal{Y}}

\newcommand{\dataspace}{\R^N}
\newcommand{\noise}{\eta}

\newcommand{\noisemeasure}{\gamma_0}

\newcommand{\indicator}{\mathbbm{1}}

\newcommand{\vfspacea}{\vfspace}
\newcommand{\vfspaceb}{V}
\newcommand{\pairSolSp}{L^2\left([0,T]\times \spatdom \right)^2}
\newcommand{\pairSol}{\tilde{\Sol}}
\newcommand{\Vdel}{\mathcal{X}_{\delta}}
\newcommand{\Vdelt}{\mathcal{X}_{\frac{\delta}{2}}}

\newcommand{\E}{\mathbb{E}}

\newtheorem{Thm}{Theorem}[section]
\newtheorem{Lem}[Thm]{Lemma}
\newtheorem{Prop}[Thm]{Proposition}
\newtheorem{Cor}[Thm]{Corollary}

\newtheorem{Def}[Thm]{Definition}
\newtheorem{Rmk}[Thm]{Remark}

\newtheorem{Not}[Thm]{Notation}

\begin{document}

\markboth{J. Borggaard, N. Glatt-Holtz, J. Krometis}
{On Bayesian Consistency for Flows Observed Through a Passive Scalar}

\maketitle

\begin{abstract}
  We consider the statistical inverse problem of estimating a background fluid flow field $\vfield$ from the partial, noisy observations of the concentration $\pdesol$ of a substance passively advected by the fluid, so that $\pdesol$ is governed by the partial differential equation
  \begin{equation*}
    \frac{\partial}{\partial t}{\pdesol}(t,\x) = -\vfield(\x) \cdot \nabla \pdesol(t,\x) + \conductivity \Delta \pdesol(t,\x) 
    \quad \text{ , } \quad \pdesol(0,\x) = \pdesol_0(\x)
  \end{equation*}
  for $t \in [0,T], T>0$ and $\x \in \spatdom=[0,1]^2$. The initial condition $\pdesol_0$ and diffusion coefficient $\conductivity$ are assumed to be known and the data consist of point observations of the scalar field $\pdesol$ corrupted by additive, i.i.d. Gaussian noise. We adopt a Bayesian approach to this estimation problem and establish that the inference is consistent, i.e., that the posterior measure identifies the true background flow as the number of scalar observations grows large. Since the inverse map is ill-defined for some classes of problems even for perfect, infinite measurements of $\pdesol$, multiple experiments (initial conditions) are required to resolve the true fluid flow. Under this assumption, suitable conditions on the observation points, and given support and tail conditions on the prior measure, we show that the posterior measure converges to a Dirac measure centered on the true flow as the number of observations goes to infinity.
\end{abstract}



\setcounter{tocdepth}{1}
\tableofcontents

\newpage

\section{Introduction}
\label{sec:introduction}\label{intro}

In this work we consider the inverse problem of estimating a background
fluid flow from partial, noisy observations of a dye, pollutant, or
other solute advecting and diffusing within the fluid. The physical model considered is the two-dimensional
advection-diffusion equation on the periodic domain
$\spatdom=[0,1]^2$:
\begin{equation} \label{eq:adr}
  \frac{\partial}{\partial t}{\pdesol}(t,\x) = -\vfield(\x) \cdot \nabla \pdesol(t,\x) + \conductivity \Delta \pdesol(t,\x) 
  \quad \text{ , } \quad \pdesol(0,\x) = \pdesol_0(\x).
\end{equation}
Here 
\begin{itemize}
\item $\pdesol: \R^{+} \times \spatdom \to \R$ is a \emph{passive
    scalar}, typically the concentration of some solute of interest,
  which is spread by diffusion and by the motion of a
  (time-stationary) fluid flow $\vfield$. This solute is ``passive''
  in that it does not affect the motion of the underlying fluid.
\item $\vfield: \spatdom \to \R^2$ is an \emph{incompressible
    background flow}, i.e., $\vfield$ is constant in time and
  satisfies $\nabla \cdot \vfield = 0$.
\item $\conductivity>0$ is the \emph{diffusion coefficient}, which
  models the rate at which local concentrations of the solute spread
  out within the solvent in the absence of advection.
\end{itemize}

We obtain finite observations $\data \in \dataspace$ subject to additive noise $\noise$, i.e.
\begin{equation} \label{eq:Gmap}
	\data = \G(\vfield)+\noise
  \quad \text{ , } \quad \noise \sim \noisemeasure
\end{equation}
for some measure $\noisemeasure$ related to the precision of the
observations. Here, the forward map $\G: \vfspace \to \dataspace$ associates the
background flow $\vfield$, sitting in a suitable function space
$\vfspace$, with a finite collection of measurements (observables) of
the resulting solution $\pdesol=\pdesol(\vfield)$ of \eqref{eq:adr}. In this work, we are primarily interested in spatial-temporal point observations:
\begin{equation}
  \G_{j}(\vfield) := \pdesol(t_j, \x_j,\vfield), \quad \text{ for any }t_j \in [0,T] 
  \text{ and } \x_j \in [0,1]^2.
\end{equation}

The goal of the inverse problem is then to estimate the flow $\vfield$ from data $\data$. The initial condition is assumed to be known, so the problem can be interpreted as
a controlled experiment, where the solute is added at known locations
and then observed as the system evolves to investigate the structure
of the underlying flow. This is a common experimental approach to investigating complex fluid flows; see, for example, 
\cite{karch2012dye,kellay2002two,merzkirch1987flow,smits2012flow}.

As we will illustrate, the inverse problem is ill-posed, i.e., the flow $\vfield$ is not 
uniquely defined by the scalar field $\pdesol$; that the observations of $\pdesol$ are both finite-dimensional and polluted by noise exacerbates this problem. We therefore adopt a Bayesian approach to 
regularize the inverse problem, as described for this problem in our companion work \cite{borggaard2018bayesian} (see also \cite{krometis2018bayesian}) and in a more general setting in, e.g., \cite{kaipio2005statistical,stuart2010inverse,dashti2017bayesian}. A key component of this approach is the selection of a prior probability measure on the space of divergence-free flows, $\vfspace$. It is then natural to ask to what extent the result of the inference depends on the choice of prior, and in particular whether the Bayesian approach to the inverse problem is \emph{consistent}: That is, under what conditions does the posterior measure concentrate on the true fluid flow as the number of observations $N$ of $\pdesol$ grows large?

In this work, we establish conditions under which the Bayesian inference of $\vfield$ given data \eqref{eq:Gmap} is consistent for i.i.d.~observational noise $\noise=\left( \noise_1, \dots, \noise_N \right), \noise_j \sim N(0,\sigma_{\noise}^2)$. We then prove that the posterior measure converges weakly to a Dirac measure centered on the true background flow as the number of scalar observations $N$ grows large; see \cref{sec:ass:proof} for a full statement of the assumptions and the key result. Here it is a nontrivial task to determine suitable conditions on the structure of the observed data and on the prior measure for which consistency would be expected to hold. As such, as a crucial starting point for the analysis of consistency, one must address difficult experimental design questions. 

In our problem, even under the noiseless and complete measurement of $\pdesol$, essential symmetries can
prevent the recovery of $\vfield$. For example, a poor choice of $\pdesol_0$ in \eqref{eq:adr} makes it impossible to distinguish between
(an infinite class of) laminar flows, so multiple experiments (initial conditions) are required to guarantee resolution of the true background flow. A second useful structural condition is that, by picking spatial-temporal observation points at random,
we can ensure a sufficiently complete recovery of the solution $\pdesol$ as the number of observation points grows. Thirdly, it is worth emphasizing that we require special conditions on the prior measure. Crucially, we identify a tail condition that ensures that flows are sufficiently smooth -- that is, the prior turns out to be critical to the result by restricting consideration to flows of limited roughness (up to a region of low probability). 

An important outcome of this experimental design is that it allows us to use compactness to effectively constrain the space of possible divergence-free velocity fields. Indeed, compactness plays an important role in two components of the consistency proof. First, we use it to show the continuity of the inverse map from $\pdesol$ to $\vfield$ (see \cref{sec:Sinv_continuous}). Second, we use it to develop a suitable uniform version of the law of large numbers in order to show that noisy observations can differentiate between the true and other scalar fields (\cref{sec:pot:concentration}).

Consistency of Bayesian estimators has been of interest since at least Laplace \cite{laplace1810memoire}, with rigorous proofs of
convergence for some problems appearing in the mid-twentieth century
\cite{doob1949application,lecam1953some}. The works \cite{freedman1963asymptotic,schwartz1965bayes,diaconis1986consistency} identified infinite-dimensional examples where Bayesian estimators are \emph{not} consistent -- that is, there are cases where the data can never guarantee recovery of the true parameter value. See, e.g., \cite{wasserman1998asymptotic}, \cite{le2000asymptotics}, or \cite{nickl2013statistical} for a more detailed description on the history of consistency and the main ideas. 

In recent years, there has been interest in extending these consistency results to infinite-dimensional inverse problems, and in particular those constrained by PDEs. Our result is one of the first on consistency in this context. Recent work in this area includes \cite{vollmer2013posterior}, which used an elliptic PDE as the guiding example, and \cite{nickl2017bernsteini}, which establishes a Bernstein-von Mises theorem -- consistency, but also contraction rates in the form of a Gaussian approximation -- for Bayesian estimation of parameters of the time-independent Schr\"{o}dinger equation.

It is worth noting that the related inverse problem of estimating the drift function $b$ from partial observations $\left\{ X_1, \dots, X_N \right\}$ of the \Ito diffusion
\begin{equation}
  dX_t = b(X_t) dt + \sigma(X_t) dW_t, \quad t>0
  \label{eq:diffusion}
\end{equation}
has been studied extensively; see, e.g., \cite{papaspiliopoulos2012nonparametric} or \cite{gugushvili2014nonparametric}. Consistency has been established in various forms for this problem; see \cite{van2013consistent,koskela2015consistency,nickl2017nonparametric,abraham2018nonparametric}. However, while the equations \eqref{eq:adr} and \eqref{eq:diffusion} are related by the Kolmogorov equations (see, e.g., \cite[Chapter 8]{oksendal2013stochastic}), the observed data are different: Observations of an individual diffusion provide an approximate measurement of the drift, whereas observations of the concentration $\pdesol$ are less direct -- movement of individual particles must be inferred. Our consistency proof therefore, while retaining some similarities with other such arguments, requires an original approach with different assumptions.

The remainder of the paper is organized as follows. \cref{sec:math_framework} describes the mathematical framework of the inverse problem and why it is ill-posed in the traditional sense. The main result and key assumptions are stated in \cref{sec:ass:proof}. Continuity of the inverse map is shown in \cref{sec:Sinv_continuous}. Uniform convergence of the log-likelihood is shown in \cref{sec:pot:concentration}. Convergence of the posterior to the inverse image of the true scalar field is shown in \cref{sec:C:true:scalar}. Finally, the proof of the main result is provided in \cref{sec:cons:conclusion}. Energy estimates for the advection-diffusion problem used to show continuity of the forward and inverse maps are reserved for \cref{sec:energyest}. 

\section{Preliminaries}
\label{sec:math_framework}

In this section, we describe the mathematical framework of the inverse
problem (\ref{eq:Gmap}). We begin by defining the functional analytic
setting for the problem, including how we represent divergence-free
background flows. We then define the inverse problem, key notation, and Bayes' Theorem 
for this application.

\subsection{Representation of Divergence-Free Background Flows}
\label{sec:modelDivFree}

The target of the inference is a divergence-free background flow
$\vfield$, so we start by describing the space $\vfspace$ of such
flows that we consider.  For this purpose we begin by recalling
the Sobolev spaces of (scalar valued) periodic functions on the 
domain $\spatdom = [0,1]^2$
\begin{equation}
  \begin{aligned}
    H^s(\spatdom) 
    &= \left\{ u : u = \sum_{\kbf \in \Z^2 \setminus \{\mathbf{0}\}} c_{\kbf} e^{2 \pi i \kbf \cdot \x }, 
    \, \overline{c_{\kbf}} = c_{-\kbf}, \, \| u\|_{H^s} < \infty \right\}, \\
    &\qquad \text{ where }  \| u\|_{H^s}^2 :=  \sum_{\kbf \in \Z^2} \norm{\kbf}^{2s} |c_{\kbf}|^2,
  \end{aligned}
    \label{eq:Hs}
 \end{equation}
 defined for any $s \in \mathbb{R}$; see
 e.g.~\cite{robinson2001infinite, temam1995navier}.  We will abuse
 notation and use the same notation for periodic divergence-free
 background flows by replacing the coefficients $c_{\kbf}$ in
 \eqref{eq:Hs} as
\begin{equation}
  c_{\kbf} = \vk \frac{\kbf^\perp}{\norm{\kbf}_2}, 
  \quad \overline{\vk}=-\vfieldnobf_{-\kbf},
  \label{eq:reality}
\end{equation}
where for $\kbf = (k_1,k_2)$ we set $\kbf^\perp = (-k_2, k_1)$ to ensure
$\kbf \cdot \kbf^\perp = 0$.  Throughout the rest of the paper we fix our parameter space as follows:
\begin{Not}[Parameter space, $\vfspace$]
  We consider background flows $\vfield \in \vfspace$, where
  $\vfspace$ is the Sobolev space (see \eqref{eq:Hs}),
  \begin{equation}
    \vfspace = H^{m}(\spatdom),  \quad \text{for some $m > 1$}
    \label{eq:vfspace}
  \end{equation}
  with coefficients $c_{\kbf}$ given by \eqref{eq:reality}.
  \label{def:vfspace}
\end{Not}
Here the exponent $m$ is chosen so that vector fields in $\vfspace$,
  as well as their corresponding solutions $\pdesol(\vfield)$, exhibit
  continuity properties convenient for our analysis below (see
  \cref{thm:th_continuous} below).  We take $L^p(\spatdom)$ with
  $p \in [1,\infty]$ for the usual Lebesgue spaces and denote the
  space of continuous and $p$-th integrable, $X$-valued functions by
  $C([0,T];X)$ and $L^p([0,T]; X)$, respectively, for a given Banach space $X$.  All
  of these spaces are endowed with their standard topologies unless
  otherwise specified.

\subsection{Mathematical Setting of the Advection-Diffusion Problem}\label{sec:math_setting}
In this section, we provide a precise definition of solutions
$\pdesol$ for the advection-diffusion problem \eqref{eq:adr}.
Crucially the setting we choose yields a map from $\vfield$ to
$\pdesol$ and then to observations of $\pdesol$ that is continuous. 

\begin{Prop}[Well-Posedness and Continuity of the solution map for
  \eqref{eq:adr}]\label{def:adr_weak}
  \mbox{}
\begin{itemize}
  \item[(i)]
  Fix any $s \geq 0$ and $m \geq s$ with $m > 0$ and suppose that
  $\vfield \in H^{m}(\spatdom)$ and $\pdesol_0 \in H^{s}(\spatdom)$.
  Then there exists a unique $\pdesol = \pdesol(\vfield, \pdesol_0)$
  such that
  \begin{align*}
    \pdesol &\in L^2_{loc}([0,\infty); H^{s+1}(\spatdom)) 
        \cap L^\infty([0,\infty); H^{s}(\spatdom)) \\
    &\quad \text{ with } \quad
    \frac{\partial \pdesol}{\partial t} 
    \in L^2_{loc}([0,\infty); H^{s-1}(\spatdom))
  \end{align*}
  so that in particular
  \begin{align*}
  \pdesol \in C([0,\infty); H^{s}(\spatdom))
  \end{align*}
  solves~(\ref{eq:adr}) at least weakly.  In other words, $\pdesol$ satisfies
    \begin{equation}
      \ip{\frac{\partial\pdesol}{\partial t}{}}{\phi}_{H^{-1}\left( \spatdom \right) \times H^1(\spatdom)}  
      + \ip{\vfield \cdot \nabla \pdesol}{\phi}_{L^2\left( \spatdom \right)} 
      + \conductivity \ip{\nabla \pdesol}{\nabla \phi}_{L^2\left( \spatdom \right)} = 0
    \label{eq:adr_weak}
  \end{equation}
  for all $\phi \in H^1(\spatdom)$ and almost all times
  $t \in [0,\infty)$.  
  \item[(ii)]
  For any $T > 0$ the map that associates $\vfield \in H^m(\spatdom)$ and
  $\pdesol_0 \in H^s(\spatdom)$ to the corresponding
  $\pdesol(\vfield, \pdesol_0)$ is continuous relative to the standard
  topologies on $H^m(\spatdom) \times H^s(\spatdom)$ and
  $C\left( [0,T]; H^s(\spatdom) \right)$.
\end{itemize}
\end{Prop}
\noindent This result can be proven using energy methods; similar
results can be found for example in~\cite{evans2010partial,
  lieberman1996second}.  In the case of a smooth solution where $s > 3$
one may also establish \cref{def:adr_weak} using particle methods as
in e.g.~\cite{oksendal2013stochastic} by observing that \eqref{eq:adr}
is the Kolmogorov equation corresponding to a stochastic differential
equation with the drift given by $\vfield$; see
\cite{krometis2018bayesian} for details in our setting.  For
completeness, we provide the {\em a priori} estimates leading to
\cref{def:adr_weak} in \cref{sec:energyest}.

\begin{Def}[Solution Operator $\Sol$, Observation Operator $\Obs$]
  Fix $\pdesol_0 \in H^s(\spatdom)$ and a time $T > 0$ and consider the phase
  space $\vfspace$ defined as (\ref{eq:vfspace}).  The forward map $\G$ as
  in \eqref{eq:Gmap} is interpreted as the composition
  $\G(\vfield) = \Obs \circ \Sol(\vfield)$, where:
  \begin{enumerate}
  \item The \emph{solution operator}
    $\Sol:\vfspace \to C([0,T]; H^s(\spatdom))$ maps a
    given $\vfield$ to the corresponding solution $\pdesol(\vfield, \pdesol_0)$ 
    of \eqref{eq:adr} (in the sense of \cref{def:adr_weak}).
  \item The \emph{observation operator}
    $\Obs: C([0,T]; H^s(\spatdom) ) \to \dataspace$ measures point observations $\Obs(\pdesol) = (\Obs_1(\pdesol), \ldots,$ $\Obs_N(\pdesol))$ defined by $\Obs_j(\pdesol) = \pdesol(t_j, \x_j)$ for $t_j \in [0,T]$ and
      $\x_j \in [0,1]^2$.
  \end{enumerate}
  \label{def:sol}
  \label{def:obs}
\end{Def}
We now note assumptions on $\vfield$ and $\pdesol_0$ under which these observations are well-defined and vary continuously
with $\vfield$.
\begin{Cor}[Continuity of $\pdesol$]
  Let $\vfield \in \vfspace$ with associated exponent $m>1$ (see \eqref{eq:vfspace}) and let 
  $\pdesol_0 \in H^s$, for $m \geq s>1$. 
  Recalling that $H^s(\spatdom), s>1$ embeds continuously in $C(\spatdom)$ in dimension $2$
  (see e.g.~\cite{robinson2001infinite}, Theorem A.1) we have that
  $C([0,T]; H^s) \subset C\left( [0,T] \times \spatdom \right)$ again with the
  embedding continuous.  Thus, with \cref{def:adr_weak}, we have that
  \begin{equation*}
    \Sol: \vfspace \to C\left( [0,T] \times \spatdom \right)
  \end{equation*}
  continuously. In particular this justifies that $\G$ is well defined
  and continuous in the case of point observations as in
  \cref{def:obs}. 
  \label{thm:th_continuous}
  \label{thm:S_continuous_L2}
  \label{thm:pt_obs_continuous}
\end{Cor}

\subsection{Bayesian Setting of the Inverse Problem}
In this subsection, we define the setting of the statistical inverse problem and note cases where the inverse map is ill-posed.  This will inform the assumptions required for the consistency argument. We close with a definition of Bayes' theorem for this problem. 

We begin by fixing some notation used
in the remainder of the paper. 

\begin{Def}[$\vtrue$, $\data$, $\G$, $\noise$]
  We frequently fix a ``true'' background flow by $\vtrue \in \vfspace$.  For the given
  $\vtrue$, the observed data $\data$ is given by
  \begin{align*}
    \data = \G(\vtrue) + \noise,
  \end{align*}
  where 
  \begin{itemize}
    \item The forward map $\G:\vfspace \to \R^N$ with $\G_{j}(\vfield) := \pdesol(t_j, \x_j;\vfield)$ corresponding to the observation point $(t_j,\x_j) \in [0,T] \times [0,1]^2$.
    \item The observational noise $\noise = \left\{ \noise_1, \dots, \noise_N \right\} \in \R^N$ for i.i.d. $\noise_j \sim N\left(0,\sigma_{\noise}^2\right)$.
  \end{itemize}
  \label{def:vtrue}
\end{Def}

We emphasize, however, that $\vtrue$ is not necessarily the only $\vfield$ that could produce such data, as we describe in the next remark.

\begin{Rmk}
  \label{rmk:Illposs}
  Since the background flow $\vfield$ enters \eqref{eq:adr} through
  the $\vfield \cdot \nabla \pdesol$ term, the inverse problem of
  recovering $\vfield$ from $\pdesol(\vfield)$ can be ill-posed.  One
  important class of examples illustrating this difficulty arises
  when $\vfield \cdot \nabla \pdesol$ is zero everywhere, in which
  case the fluid flow does not have any effect on $\pdesol$. Two such
  examples are as follows:
\begin{itemize}
\item[(i)] \underline{Ill-posedness: Laminar Flow:}
 Let $\pdesol_0(\x)$ be independent of $x_2$ and
    $\vtrue = \left( 0,f(x_1) \right)$. Then
    $\pdesol(\vtrue)=\pdesol(\vfield)$ for any
    $\vfield=\left( 0,g(x_1) \right)$.
  \item[(ii)] \underline{Ill-posedness: Radial Symmetry:} Set
    $\pdesol_0(\x)\propto \sin (\pi x_1) + \sin(\pi x_2)$ and
    $\vtrue = (\cos(\pi x_2), -\cos(\pi x_1))$. Then
    $\pdesol(\vtrue)=\pdesol(\vfield)$ for any
    $\vfield = c \vtrue, \,c \in \R$.
\end{itemize}
\noindent In these cases, even noiseless and complete spatial/temporal
observations of $\pdesol$ have no way to discriminate between a range
of background flows, making it impossible to uniquely identify a true
background flow $\vtrue$ in general.
\end{Rmk}

We have following adaptation of Bayes' Theorem to the
advection-diffusion problem; see the derivation in \cite[Appendix C]{borggaard2018bayesian} or \cite{dashti2017bayesian} for additional information.
\begin{Thm}[Bayes' Theorem]
	\label{thm:Bayes:AD}
  Fix a prior distribution $\mu_0 \in \mbox{Pr}(\vfspace)$ and let forward maps $\G_j$, data $\data_j$, and associated i.i.d. observational noise $\noise_j \sim N(0,\sigma_{\noise}^2)$ be as defined in \cref{def:vtrue}. Then the \emph{posterior} measure $\mu_\data$ associated with the random variable $\vfield | \data$ is absolutely continuous with respect to $\mu_0$ and given by
  \begin{equation}
      \mu_\data(d\vfield) 
      = \frac{1}{Z_{\data}} 
      \exp \left[ -\frac{1}{2\sigma_\noise^2} \sum_{j=1}^N 
        \left(\data_j-\G_j\left(\vfield\right) \right)^2\right] 
      \mu_0(d\vfield) 
    \label{eq:bayes}
  \end{equation}
  where $Z_{\data}$ is the normalization
  \begin{equation} \label{eq:bayes_gt0_cond}
    Z_{\data} = \int_{\vfspacea} 
    \exp \left[ -\frac{1}{2\sigma_\noise^2} 
    \sum_{j=1}^N \left(\data_j-\G_j\left(\vfield\right) \right)^2\right] 
       \mu_0(d\vfield).
  \end{equation}
\end{Thm}

\section{Statement of the Main Result}
\label{sec:ass:proof}

With the mathematical preliminaries in \cref{sec:math_framework} in hand, we are now ready to provide a precise formulation of the main result of the paper. Referring back to \cref{rmk:Illposs} we do not expect consistency to hold
without delicate assumptions on the initial conditions in
(\ref{eq:adr}) and on the observation points in our forward function
$\G$ in (\ref{eq:Gmap}).  Moreover our result relies on the selection
of an appropriate prior $\mu_0$.  In particular this $\mu_0$ should
distinguish the regularity of the `true' background flow $\vtrue$ 
for which we assume there is greater degree of spatial smoothness
than for generic elements in the ambient parameter space $\vfspace$ (although a slight generalization is described in \cref{rmk:alt:assumption}). We therefore define
an additional smaller space used throughout.
\begin{Def}[Higher Regularity Space]
  Define the space
  \begin{equation}
    \vfspaceb = H^{m^\star}\!( \spatdom ), \quad m^\star > m,
    \label{eq:vfspaceb}
  \end{equation}
  where $m$ is the exponent associated with the parameter space
  $\vfspace$ defined according to \eqref{eq:vfspace}. We denote 
  $\| \cdot \|_V$ for the associated norm and take
  \begin{equation}
    B_{\vfspaceb}^r(\vfield_0) 
    = \left\{ \vfield \in \vfspaceb : \norm{ \vfield - \vfield_0 }_\vfspaceb \le r \right\}.
    \label{eq:bvfspaceb}
  \end{equation}
  \label{def:vspaces}
  i.e.~the ball about $\vfield_0 \in \vfspaceb$ of radius $r>0$ in the $\vfspaceb$-norm.
\end{Def}

Our main result is as follows
\begin{Thm}[Convergence of Posterior to a Dirac]
  \label{thm:post:conv}
  Let $\left\{(t_j,\x_j)\right\}_{j=1}^\infty$ be a sequence of
  observation points that we assume are i.i.d.~uniform random variables in $[0,T] \times \spatdom$.
  Fix any $\pdesol_0^{(1)},\pdesol_0^{(2)} \in H^m$, with $m > 1$
  determined from \eqref{eq:vfspace}, such that
  \begin{equation}
    \left( \nabla \pdesol_0^{(1)}(\x) \right)^{\perp} \cdot \nabla \pdesol_0^{(2)}(\x) \ne 0,
    \quad \text{for almost all  } \x \in \spatdom.
    \label{eq:assum:multi_ic}
  \end{equation}
  Define the parameter-to-observable (forward) maps $\G_j$ for
  $\left\{(t_j,\x_j)\right\}_{j=1}^\infty$ and the initial conditions
  $\pdesol_0^{(1)},\pdesol_0^{(2)}$ by
  \begin{equation}
    \label{eq:for:m:consist}
    \begin{aligned}
      \G_{2j-1}\left( \vfield \right) &:=
           \pdesol(t_{j},\x_{j},\vfield, \pdesol_0^{(1)})\\
      \G_{2j}\left( \vfield \right) &:=
           \pdesol(t_{j},\x_{j},\vfield, \pdesol_0^{(2)})\\
    \end{aligned}
  \end{equation}
  for $j = 1, 2, \dots$.
  As in \cref{def:vtrue}, we fix any $\vtrue \in \vfspaceb$
  and draw data points $\left\{ \data_j \right\}_{j=1}^\infty$, where
  \begin{align}
    \data_j = \G_j\left(\vtrue\right) + \noise_j
    \label{eq:data:draw:rand}
  \end{align}
  for i.i.d.~observational noises
  $\noise_j \sim N\left(0,\sigma_\noise^2\right)$ that are
  independent of the observation points
  $\left\{(t_j,\x_j)\right\}_{j=1}^\infty$.

  Fix a prior distribution $\mu_0 \in \mbox{Pr}(\vfspace)$ and 
  for $N \geq 1$ observations, let $\mu_\data^N$ be the Bayesian
  posterior measure on $\vfspacea$, given by (cf. \cref{thm:Bayes:AD})
  \begin{align}
    \mu_\data^N(d\vfield) 
    = \frac{1}{Z^N_{\data}} 
    \exp \left[ -\frac{1}{2\sigma_\noise^2} \sum_{j=1}^N 
      \left(\data_j-\G_j\left(\vfield\right) \right)^2\right] 
    \mu_0(d\vfield) 
    \label{eq:rand:post:measure}
  \end{align}
  where $Z^N_{\data}$ is the normalization
  \begin{align*}
    Z^N_{\data} = \int_{\vfspacea} 
    \exp \left[ -\frac{1}{2\sigma_\noise^2} 
    \sum_{j=1}^N \left(\data_j-\G_j\left(\vfield\right) \right)^2\right] 
       \mu_0(d\vfield).
  \end{align*}
  Suppose that 
  \begin{align}
    \text{ for any $r > 0$, $\mu_0\left( B_\vfspaceb^{r}\left( \vtrue \right) \right)>0$.}
    \label{eq:strng:spt:cond}
  \end{align}
  Additionally assume that there exists an $f:\R^+ \to \R^+$ such that $f$ is
  monotone increasing with $\lim_{r \to \infty} f(r) = \infty$ and
      \begin{align}
        \sup_N \int_\vfspacea f\left( \norm{\vfield}_\vfspaceb \right) 
        \mu_\data^N(d\vfield) < \infty \quad a.s.
          \label{eq:assum:prior_tail_2}
      \end{align}
  Then $\mu_\data^N \rightharpoonup \delta_{\vtrue}$ (weakly in $\vfspace$) almost surely.
  In other words, on a set of full measure,
\begin{align}
  \int_{\vfspace} \phi( \vfield) \mu^N_\data(d \vfield) \to 
  \phi(\vtrue) 
  \quad \text{as $N \to \infty$ for any $\phi \in C_b(\vfspace)$}.
  \label{eq:weak:conv:def}
\end{align}
\end{Thm}
\begin{Rmk}[Support of the prior]
  We note that the assumption \eqref{eq:strng:spt:cond} is a classic assumption in posterior consistency, cf. \cite[Chapter 6]{ghosal2017fundamentals}; if the prior ``rules out'' the true flow, the posterior cannot recover it. 
\end{Rmk}

\begin{Rmk}[Equivalent prior tail condition]
  \label{rmk:prior:eq:cond}
  Condition \eqref{eq:assum:prior_tail_2} is equivalent to the assumption that for all $\epsilon > 0$ there exists an $R$ such that
  \begin{align}
      \sup_N \mu_\data^N\left( \left( B_\vfspaceb^{R}\left(\mathbf{0}\right)
        \right)^c \right) < \epsilon.
     \label{eq:assum:prior_tail_3}
  \end{align}
  To establish that \eqref{eq:assum:prior_tail_2} implies \eqref{eq:assum:prior_tail_3}, let $\epsilon > 0$ and choose $R$ such that
  \begin{align*}
    \sup_N \int_\vfspacea f\left( \norm{\vfield}_\vfspaceb \right) \mu_\data^N(d\vfield) 
    < \epsilon f(R).
  \end{align*}
  Then for any $N \geq 1$, Markov's inequality yields
  \begin{align*}
      \mu_\data^N\left( \left( B_\vfspaceb^{R}\left(\mathbf{0}\right)
        \right)^c \right) = \int_{\left(
          B_\vfspaceb^{R}\left(\mathbf{0}\right) \right)^c}
      \mu_\data^N(d\vfield) \le \frac{1}{f(R)} \int_{\vfspacea}
      f(\norm{\vfield}_\vfspaceb) \mu_\data^N(d\vfield) < \epsilon.
  \end{align*}
  Thus, \eqref{eq:assum:prior_tail_2} implies \eqref{eq:assum:prior_tail_3}. For the converse direction, use \eqref{eq:assum:prior_tail_3} to select an increasing sequence $\{R_j\}_{j=1}^\infty$ such that
  \begin{equation*}
    \sup_N \mu_\data^N\left(\left(B_\vfspaceb^{R_j}(0)\right)^c\right) < \frac{1}{4^{j}}.
  \end{equation*}
  Now define $f(r) = \sum_{j=1}^\infty 2^{j} \indicator_{r \in [R_{j},R_{j+1})}$. Then
  \begin{align*}
    \sup_N \Exp_{\mu_{\data}^N} f(\norm{\vfield}_\vfspaceb) 
    &= \sup_N \Exp_{\mu_{\data}^N} \sum_{j=1}^\infty 2^{j} \indicator_{\norm{\vfield}_\vfspaceb \in [R_{j},R_{j+1})} 
    \le \sup_N \Exp_{\mu_{\data}^N} \sum_{j=1}^\infty 2^{j} \indicator_{\vfield \in \left(B_\vfspaceb^{R_j}(0)\right)^c}  \\
    &= \sum_{j=1}^\infty 2^{j} \sup_N \mu_{\data}^N\left((B_\vfspaceb^{R_j}(0))^c\right) 
    < \sum_{j=1}^\infty 2^{j} \frac{1}{4^{j}} 
    = \sum_{j=1}^\infty \frac{1}{2^{j}} 
    < \infty
  \end{align*}
  so that \eqref{eq:assum:prior_tail_3} implies \eqref{eq:assum:prior_tail_2}. 
\end{Rmk}
\begin{Rmk}[Sufficient conditions on the prior]
  \label{rmk:prior:const}
  Suppose that
  \begin{align}
    \mu_0\left( B_\vfspaceb^r\left( \mathbf{0} \right) \right) = 1,
    \label{eq:assum:prior_tail_1}
  \end{align}
  for some $r>0$.  In this case, \eqref{eq:assum:prior_tail_3} and therefore \eqref{eq:assum:prior_tail_2} (see \cref{rmk:prior:eq:cond}) are clearly satisfied. Thus we can guarantee the existence
  of a class of non-trivial priors such that \cref{thm:post:conv}
  holds.  On the other hand the reverse implication is not to be
  expected to hold and thus the general significance of
  \eqref{eq:assum:prior_tail_2} for the admissible classes of $\mu_0$
  is not immediately clear.  In particular $\mu_0$ having bounded
  support is a strong restriction and indeed we conjecture that
  there is a class of Gaussian measures on $V$ such that
  \eqref{eq:assum:prior_tail_2} still holds.  We will investigate this
  question in future work.
\end{Rmk}

\begin{Rmk}[Poincar\'e inequality, support of $\mu_0$]
  \label{rmk:PC:Supp}
  Since we are assuming that elements in $\vfspace$ are mean-free (see \eqref{eq:Hs})
  we have the Poincar\'e-type inequality
  \begin{align}
    \|\vfield\|_\vfspace \leq C \|\vfield\|_\vfspaceb,
    \label{eq:Poincare}
  \end{align}
  for a constant $C$ independent of $\vfield$.  As such, for any $\epsilon > 0$,
  $B_\vfspaceb^{\epsilon} \subset B_\vfspace^{C\epsilon}$ where $C$ is the constant
  appearing in \eqref{eq:Poincare}.  In particular under the 
  condition \eqref{eq:strng:spt:cond} in \cref{thm:post:conv}
  we have that $\vtrue \in \mbox{supp}(\mu_0)
  = \{ \vfield \in \vfspace: \mu_0(B_\vfspace^r(\vtrue)) > 0, \text{ for all } r > 0\}$.
\end{Rmk}

\begin{Rmk}[Restrictions on the initial conditions]
  \label{rmk:IC:reasons}
  It unavoidable that that we impose a condition such as
  (\ref{eq:assum:multi_ic}) on the initial data in
  \cref{thm:post:conv}.  In \cref{rmk:Illposs} we provide two examples
  where the observations have no way to discriminate between a range
  of background flows.  In these two examples as well as many other
  classes of initial conditions, the posterior fails to concentrate on
  $\vtrue$ as the number of observations $N \to \infty$ (except for
  very particular priors).  An interesting question for future
  work is to characterize the support of the limiting measure for the
  analogue of $\mu^N_\data$ as $N \to \infty$ as a function of a
  single initial condition $\pdesol_0$.
\end{Rmk}

\begin{Rmk}[The role of $T$]
  \label{rmk:time}
  It is worth noting that any $T > 0$ suffices for consistency. This is because the two initial conditions have been chosen so that any $\vfield$ will have an immediate effect on at least one of $\pdesol^{(1)}, \pdesol^{(2)}$. As a result, we need only observe the evolution of the two systems for some non-zero time interval to identify the effect of $\vtrue$ (as the number of observations grows large). The question of how $T>0$ (and hence the placement of observation points) affects the rate at which $\mu_\data^N$ converges to $\delta_{\vtrue}$ is a much more delicate question for future work.
\end{Rmk}

\begin{Rmk}[Case where $\vtrue \not \in \vfspaceb$]
  \label{rmk:alt:assumption}
  As pointed out by a helpful reviewer, \cref{thm:post:conv} can be extended to the case where $\vtrue \not\in \vfspaceb$ by replacing $B_\vfspaceb^{r}\left( \vtrue \right)$ with $\vtrue + B_\vfspaceb^{r}\left( 0 \right)$ throughout the proof, as long as the assumption \eqref{eq:assum:prior_tail_2} is appropriately recentered on $\vtrue$, i.e.
  \begin{align*}
    \sup_N \int_\vfspacea f\left( \norm{\vfield-\vtrue}_\vfspaceb \right) 
    \mu_\data^N(d\vfield) < \infty \quad a.s.
  \end{align*}
  or, as in \cref{rmk:prior:eq:cond}, for all $\epsilon > 0$ there exists $R$ such that
  \begin{align*}
    \sup_N \mu_\data^N((\vtrue+B_V^R(0))^c) < \epsilon.
  \end{align*}
  That is, the prior needs to help rule out flows that are far from $\vtrue$ as measured by $\norm{\cdot}_\vfspaceb$.
\end{Rmk}

Before turning to the technical details let us provide an overview of
the method of the proof of \cref{thm:post:conv}.  Our starting point
is based on two basic observations.  Firstly, according to
Portmanteau's Theorem, condition \eqref{eq:weak:conv:def}
can be established with the equivalent condition that
\begin{align}
  \liminf_{N \geq 1} \mu^N_\data( B_H^\epsilon(\vtrue)) \geq 1
  \label{eq:conv:suff:cond}
\end{align}
for any $\epsilon > 0$.  See e.g.~\cite{billingsley2013convergence}
for further details on such generalities concerning the weak 
convergence of probability measures.

Our second observation concerns using the law of large numbers
to identify the approximate character of
the potential terms in \eqref{eq:rand:post:measure} for large
$N$. Referring back to \eqref{eq:data:draw:rand} and \eqref{eq:rand:post:measure}, we have
\begin{align*}
 \left(\data_j-\G_j\left(\vfield\right) \right)^2
  = \eta_j^2 + 2\eta_j (\G_j(\vtrue) -\G_j(\vfield )) + \left(\G_j(\vtrue) -\G_j(\vfield ) \right)^2.
\end{align*}
Invoking the law of large numbers, using assumed statistical properties 
of $\{\eta_j\}_{j \geq 1}$ and $\{(t_j,x_j)\}_{j \geq 1}$, we have
\begin{align}
  \frac{1}{N}\sum_{j =1}^N &\left(\data_j-\G_j\left(\vfield\right) \right)^2
  \approx  \sigma_\eta^2 
  + \frac{1}{T} \sum_{l =1}^2 \int_0^T\int_{\spatdom} 
  \left(\pdesol(t,x, \vfield, \pdesol_0^{(l)})
      - \pdesol(t,x, \vtrue, \pdesol_0^{(l)})\right)^2 d\x dt
      	\label{eq:LLN:sloppy}
\end{align}
for all $N$ sufficiently large.\footnote{Referring back to
  \cref{sec:modelDivFree} we are assuming that $\spatdom$ is unit
  length.}  For $\delta > 0$, take
\begin{align}
  \label{eq:forward:th:ball}
  \Vdel
  = \left\{ \vfield \in \vfspacea
  : \sum_{l =1}^2 \int_0^T \int_{\spatdom} 
  \left(\pdesol(t,x, \vfield, \pdesol_0^{(l)})
      - \pdesol(t,x, \vtrue, \pdesol_0^{(l)})\right)^2 d\x dt \!
  < \delta^2   \right\}.
\end{align}
Invoking \eqref{eq:LLN:sloppy}, we observe that
\begin{align}
  \mu^N_\data(\Vdel^c)
  &\approx \frac{{ \displaystyle \int_{\Vdel^c}} \exp\left[
     -\frac{N}{4\sigma_\noise^2 T} \displaystyle \sum_{l =1}^2 \int_0^T\int_{\spatdom} 
          \left(\pdesol(\pdesol_0^{(l)}, \vfield)
  - \pdesol(\pdesol_0^{(l)}, \vtrue)\right)^2 d\x dt  \right]
         \mu_0(d \vfield)}
               {{\displaystyle\int} \exp\left[
     -\frac{N}{4\sigma_\noise^2 T} \displaystyle \sum_{l =1}^2 \int_0^T\int_{\spatdom} 
          \left(\pdesol(\pdesol_0^{(l)}, \vfield)
  - \pdesol(\pdesol_0^{(l)}, \vtrue)\right)^2 d\x dt  \right]
    \mu_0(d \vfield)}
  \notag\\
 &\leq \frac{\exp(-\frac{N \delta^2}{4\sigma_\noise^2 T})\mu_0(\Vdel^c) }
   {{\displaystyle \int_{\mathcal{X}_{\delta/2}}} \! \! \! \exp\left[
      -\frac{N}{4\sigma_\noise^2 T} \displaystyle \sum_{l =1}^2 \int_0^T\int_{\spatdom} 
           \left(\pdesol(\pdesol_0^{(l)}, \vfield)
   - \pdesol( \pdesol_0^{(l)}, \vtrue)\right)^2 d\x dt  \right]
   \mu_0(d \vfield)}
   \notag\\
  &\leq \frac{\exp(- \frac{3N\delta^2}{16\sigma_\noise^2T})}{\mu_0(\mathcal{X}_{\delta/2})}.
    \label{eq:cnst:mn:bnd}
\end{align}
Here note that (cf. \cref{rmk:PC:Supp}) $\vtrue \in \mbox{supp}(\mu_0)$
so that we are not dividing by zero in the final upper bound.

One is thus tempted to now combine \eqref{eq:conv:suff:cond} and
\eqref{eq:cnst:mn:bnd} 
and find for every $\epsilon > 0$ a corresponding $\delta > 0$ such that 
\begin{align}
  \Vdel \subset B_\vfspace^\epsilon
  \label{eq:supset}
\end{align}
so that 
\begin{align*}
  \mu_\data^N \left( B_\vfspace^\epsilon \right) \ge
  \mu_\data^N \left( \Vdel \right) \ge
  1- \frac{\exp(- \frac{3N\delta^2}{16\sigma_\noise^2T})}{\mu_0(\mathcal{X}_{\delta/2})},
\end{align*}
yielding the desired weak convergence
\eqref{eq:weak:conv:def}.  However this na\"ive argument
runs up against two fundamental flaws:
\begin{itemize}
\item[(i)] Although, as we establish below in \cref{thm:S_1to1}, the
  condition \eqref{eq:LLN:sloppy} ensures that the map
  $\vfield \mapsto (\pdesol(\cdot, \pdesol^{(1)}_0, \vfield),
  \pdesol(\cdot, \pdesol^{(2)}_0, \vfield))$
  is injective into $L^2([0,T] \times \spatdom)$ it is not clear if
  this map has a continuous inverse, which we would need for \eqref{eq:supset}.
\item[(ii)] It is not obvious that we have sufficient uniformity over
  $\vfield \in H$ in our invocation of the LLN in
  \eqref{eq:LLN:sloppy}.  In particular this means that the
  approximation in the first line in \eqref{eq:cnst:mn:bnd} would be
  unjustified.
\end{itemize}

We address both of these concerns by assuming a little bit of extra
regularity for our `true vector field' taking $\vtrue \in \vfspaceb$ and by
making effective use of the prior to identify this regularity for
$\vtrue$ (see assumptions \eqref{eq:strng:spt:cond},
\eqref{eq:assum:prior_tail_2}).  With the Rellich-Kondrachov theorem
we are thus able to use `compactness' to address both concerns.
Indeed although an injective, continuous map $\psi$ does not have a
continuous inverse in general, this property does hold true when the
domain of $\psi$ is compact; see \cref{thm:finv_continuous} below.
Regarding the second concern, we establish a uniform version of
the LLN \cref{thm:ulln} below (and see also \cite{newey1994large,nickl2013statistical}) but
our proof makes essential use of the fact that the `parameter' (which 
for us is $\vfield \in H$) lies in
a compact set.

The precise proof of \cref{thm:post:conv} is presented in a series of
sections as follows.  Firstly in \cref{sec:Sinv_continuous} we
address the injectivity of the forward map under
(\ref{eq:assum:multi_ic}) as well as continuity of the inverse map.  
In \cref{sec:pot:concentration} we introduce
a uniform version of the Law of Large Numbers, \cref{thm:ulln} and use
this Proposition to obtain a quantitative version 
of \eqref{eq:LLN:sloppy}.  \cref{sec:C:true:scalar} establishes that
$\mu^N_\data$ converges on the `true scalar field' $\pdesol(\vtrue)$ as
$N \to \infty$.  Finally \cref{sec:cons:conclusion} uses the machinery
now in place to complete the proof of \cref{thm:post:conv}.

\section{Continuity of Inverse Map}
\label{sec:Sinv_continuous}

In this section, we lay out conditions under which the inverse
solution map $\pdesol \mapsto \vfield$ is continuous.  This requires
some care.  Indeed it is not true in general that the forward map
$\Sol$ is injective, as illustrated in \cref{rmk:Illposs}.  As such,
counterexamples to \cref{thm:post:conv} exist
(cf. \cref{rmk:IC:reasons}) if we fail to impose a suitable assumption
on the initial condition(s) for \eqref{eq:adr} \`a la
\eqref{eq:assum:multi_ic}.

With this in mind we now define the solution map associated with the
solution of \eqref{eq:adr} for the multiple initial conditions.
\begin{Not}[Paired solution map $\pairSol$]
   \label{def:S_pair}
  Fix any $\pdesol_0^{(1)}, \pdesol_0^{(2)} \in H^m$ for $m> 1$ as in
  \eqref{eq:Hs} and let $\pdesol^{(1)}(\vfield), \pdesol^{(2)}(\vfield)$ be the
  associated solutions of \eqref{eq:adr} corresponding to $\vfield \in H$ defined according to
  \cref{def:adr_weak}.  We denote 
  \begin{align*}
  \pairSol(\vfield) = 
    \left(\pdesol(\cdot, \vfield, \pdesol^{(1)}_0),
       \pdesol(\cdot, \vfield, \pdesol^{(2)}_0)\right),
  \end{align*}
  regarding $\pairSol$ as a map $\pairSol:\vfspacea \to \pairSolSp$.
\end{Not}

We now observe that the the paired solution map $\pairSol$ is continuous
(\cref{thm:S_continuous_pair}) and that under condition \eqref{eq:assum:multi_ic}, 
$\pairSol$ is 1-to-1 (\cref{thm:S_1to1}). 

\begin{Cor}[$\pairSol$ continuous]
  The paired solution map $\pairSol:\vfspacea \to \pairSolSp$ (see
  \cref{def:S_pair}) is continuous.
  \label{thm:S_continuous_pair}
\end{Cor}
\begin{proof}
  For any $\pdesol_0 \in H^m$ (with $m$ as in \eqref{eq:vfspace}) the
  associated solution map
  $\Sol:\vfspacea \to L^2\left([0,T]\times \spatdom \right)$ given by
  $\Sol(\vfield) = \pdesol(\cdot, \vfield, \pdesol_0)$ is continuous
  by \cref{thm:S_continuous_L2} so that the map $\pairSol$ is also
  continuous.
\end{proof}

\begin{Lem}[$\pairSol$ injective]
  Let $\pairSol$ be the paired solution map given in \cref{def:S_pair}
  with initial conditions satisfying \eqref{eq:assum:multi_ic}.  Suppose
  that $\vfield, \tilde{\vfield} \in \vfspacea$ such that
  \begin{equation}
    \norm{\pairSol(\vfield) - \pairSol(\tilde{\vfield})}_{\pairSolSp}=0.
    \label{eq:pairSol_equal_l2}
  \end{equation}
  Then $\vfield = \tilde{\vfield}$, or in other words, $\pairSol$ is 
  injective.
  \label{thm:S_1to1}
\end{Lem}
\begin{proof}
  Let $\vfield, \tilde{\vfield} \in \vfspacea$ satisfy \eqref{eq:pairSol_equal_l2}, i.e., 
  \begin{align*}
    \norm{ \pdesol^{(i)}(\cdot,\vfield) - \pdesol^{(i)}(\cdot,\tilde{\vfield}) }_{L^2\left([0,T]\times \spatdom \right)}=0, 
    \quad i=1,2.
  \end{align*}
  Then
  $\pdesol^{(i)}(t,\x,\vfield) = \pdesol^{(i)}(t,\x,\tilde{\vfield})$
  for almost all $t,\x$ and $i=1,2$. 
  Denote $\pdesol^{(i)}(t,\x) \coloneqq 
  \pdesol^{(i)}(t,\x,\vfield) =
  \pdesol^{(i)}(t,\x,\tilde{\vfield})$.
  Then from \cref{def:adr_weak}, $\pdesol^{(i)}(t,\x)$ satisfies
  \begin{equation}
      \ip{\frac{\partial\pdesol^{(i)}}{\partial t}{}}{\phi}_{H^{-1}\left( \spatdom \right) \times H^1(\spatdom)}  
      + \ip{\bfU \cdot \nabla \pdesol^{(i)}}{\phi}_{L^2\left( \spatdom \right)} 
      + \conductivity \ip{\nabla \pdesol^{(i)}}{\nabla \phi}_{L^2\left( \spatdom \right)} = 0
  \end{equation}
  for all $\phi \in H^1(\spatdom)$, almost all time $t \in [0,\infty)$, $i=1,2$, and $\bfU = \vfield, \tilde{\vfield}$.  
  Subtraction leads to
  \begin{equation}
    g(t) = \ip{ \left( \vfield - \tilde{\vfield} \right) \cdot \nabla \pdesol^{(i)}(t)}{\phi}_{L^2\left( \spatdom \right)} = 0
  \end{equation}
  for all $\phi \in H^1(\spatdom)$, almost all time $t \in [0,\infty)$, and $i=1,2$. Since we also have
  \begin{equation}
    g(t) = -\ip{ \pdesol^{(i)}(t)}{ \left( \vfield - \tilde{\vfield} \right) \cdot \nabla \phi}_{L^2\left( \spatdom \right)}
  \end{equation}
  and $\pdesol^{(i)} \in C([0,\infty); L^2(\spatdom))$ by \cref{def:adr_weak} we infer that $g(t)=0$ for all $t \ge 0$ and in particular $g(0)=0$.
  Therefore
  \begin{equation}
    \left( \vfield(\x) - \tilde{\vfield}(\x) \right) \cdot \nabla \pdesol_0^{(i)}(\x) = 0
  \end{equation}
  for $i=1,2$ and almost all $\x \in \spatdom$. However, under the assumption \eqref{eq:assum:multi_ic}, 
  $\nabla \pdesol_0^{(1)}(\x),\nabla \pdesol_0^{(2)}(\x)$ span $\R^2$
  at almost all $\x$. Therefore $\tilde{\vfield}(\x) = \vfield(\x)$
  for almost all $\x$ and hence
  $\norm{ \vfield - \tilde{\vfield}}_\vfspacea = 0$, completing the proof.
\end{proof}

Even under the conditions of \cref{thm:S_1to1} it remains unclear if
$\pairSol$ has a continuous inverse.  To remedy this we recall the
following elementary fact from real analysis suggesting we 
further restrict the domain of $\pairSol$.
\begin{Lem}
  Let $Y,Z$ be metric spaces and suppose that $B \subset Y$ is
  compact. Let $f: Y \to Z$ be injective and continuous. Then
  $f^{-1}:f(B) \to Y$ is also continuous.\footnote{Here, we denote
  $f(B)\coloneqq \left\{ f(y) \in Z : y \in B \right\}$.}
  \label{thm:finv_continuous}
\end{Lem}
We provide a proof here for completeness, although this result can be found in many real analysis references; see, e.g., \cite[Theorem 4.17]{rudin1964principles}. 
\begin{proof}
  Let $z_n,z \in f(B)$ such that $z_n \to z$. Define $y_n,y \in B$
  according to $y_n = f^{-1}(z_n)$ and $y = f^{-1}(z)$. We 
  would like to show that $y_n \to y$ as $n \to \infty$. 
  
  To this end let $y_{n'}$ be any subsequence.  Since $B$ compact,
  there exists a subsubsequence $y_{n''}$ that converges in $B$;
  denote this limit $\tilde{y} \in B$. Then, since $f$ continuous,
  $f(y_{n''}) \to f(\tilde{y})$. But, by definition and the assumed
  convergence of $z_n$, we also have
  $f(y_{n''}) = z_{n''} \to z = f(y)$ so that
  $f(\tilde{y})=f(y)$. Since $f$ injective, $y=\tilde{y}$, i.e.,
  $f^{-1}(z_{n''}) \to f^{-1}(z)$. However since the original
  subsequence was arbitrary we have in fact that $f^{-1}(z_{n}) \to f^{-1}(z)$
  yielding the desired result.
\end{proof}

From \cref{thm:finv_continuous} we draw the following two
conclusions.
\begin{Cor}[$\pairSol^{-1}$ continuous]
  \label{thm:Sinv_continuous}
  Let $\pairSol:\vfspacea \to \pairSolSp$ be the paired solution map
  given in \cref{def:S_pair} with initial conditions meeting
  \eqref{eq:assum:multi_ic}. Then, for any $r > 0$ and
  $\vfield_0 \in \vfspaceb$,
  $\pairSol^{-1}:\pairSol\left(B_\vfspaceb^r\left( \vfield_0
    \right)\right) \to \vfspacea$ is continuous.
\end{Cor}
\begin{proof}
  We have $\pairSol:\vfspacea \to \pairSolSp$ continuous by
  \cref{thm:S_continuous_pair} and injective by \cref{thm:S_1to1}. We
  also have $B_\vfspaceb^r\left( \vfield_0 \right)$ compact in
  $\vfspacea$ by the Rellich-Kondrachov Theorem; see, e.g., Corollary
  A.5 of \cite{robinson2001infinite}. Therefore,
  $\pairSol^{-1}:\pairSol\left(B_\vfspaceb^r\left( \vfield_0
    \right)\right) \to \vfspacea$
  is continuous by \cref{thm:finv_continuous}.
\end{proof}

\begin{Cor}
  \label{thm:Sinv_continuous2}
  Let $r > 0$ and $\vfield_0 \in \vfspaceb$. For all $\epsilon > 0$, there exists a $\delta > 0$ such that 
  \begin{align*}
    \left\{ \vfield \in \vfspace : 
      \norm{ \tilde{\Sol}(\vfield) - \tilde{\Sol}(\vfield_0 ) }_{\pairSolSp} 
      < \delta \right\} 
    \cap
    B_\vfspaceb^r\left(\vfield_0  \right) 
        \subset 
            B_\vfspacea^\epsilon\left( \vfield_0 \right) .
  \end{align*}
\end{Cor}

\section{Concentration of Normalized Potentials, 
  Uniform Law of Large Numbers}
\label{sec:pot:concentration}

The next step in our analysis is to prove a rigorous and more
quantitative version of \eqref{eq:LLN:sloppy},
\cref{thm:l2_decompose_unif}, which yields the asymptotics of the
potential functions (log-likelihoods) appearing in the posterior
measures $\mu^N_\data$ defined as in \eqref{eq:rand:post:measure}.  As
a preliminary step we introduce a uniform version of the Law of Large
Numbers. See also \cite{newey1994large,nickl2013statistical} for previous related results.

\begin{Prop}[Uniform Law of Large Numbers]
  \label{thm:ulln}
  Let $(X, \Xmet)$ be a metric space with $B \subset X$ compact and
  $f:\R^n \times X \to \R$ (Borel) measurable. Take
  $\left\{ Z_j \right\}_{j=1}^\infty \in \R^n$ to be an i.i.d.~sequence
  of random variables and let $Z$ be any random variable
  with this distribution. Assume that 
  \begin{align}
    \Exp f(Z,x)^2 < \infty, \quad
    \text{ for all } x \in B 
    \label{eq:ulln_moment}
  \end{align}
  and that there exists a deterministic function
  $d:\R^n \to \R^{+}$ with $\Exp d(Z)^2 < \infty$ such that for all
  $\epsilon > 0$ and $x \in B$, there exists a $\delta = \delta(x,\epsilon) > 0$ such that
  \begin{align}
    \Xmet(x,\tilde{x}) < \delta \implies 
    |f(z, x) -f(z,\tilde{x})| \le d(z)\epsilon, \text{ for all $z \in \R^n$.}
    \label{eq:ulln_continuity}
  \end{align}
  Then 
  \begin{align}
    \lim_{N\to\infty} \sup_{x \in B} \left| 
    \frac{1}{N} \sum_{j=1}^N f(Z_j,x) - \Exp f(Z,x) \right| = 0 \quad a.s.
    \label{eq:ulln}
  \end{align}
\end{Prop}
\begin{proof}
  Note that, since $d$ is non-negative, $\Exp d(Z)=0$ implies that 
  \begin{align*}
  	\tilde{\Omega} = \bigcap_{j=1}^\infty\left\{ \omega \in \Omega : d(Z_j(\omega)) =0\right\}
   \end{align*}
   is a set of full measure in which case the random functions $x \mapsto f(Z_j, x)$, 
   $j =1,2, \ldots$ are all constant on $\tilde{\Omega}$ and the result \eqref{eq:ulln} 
   follows for this special case. 
   
   We turn to the nontrivial case where $\Exp d(Z) \ne 0$.
  Define $g(z,x) \coloneqq f(z,x) - \Exp f(Z,x)$, $z \in \R^n, x \in X$. 
  Then by our assumptions
  on $f$, $\Exp g(Z,x)^2 < \infty$ for every $x \in B$.  Note also that
  for any $x, \tilde{x} \in X$, $\epsilon > 0$, and $z \in \R^n$,
  \begin{align}
    |f(z,x) -f(z,\tilde{x})&| \le d(z)\epsilon \implies |g(z,x)-g(z,\tilde{x})| \le \left[ d(z) + \Exp d(Z) \right]\epsilon.
    \label{eq:ulln_g_continuity}
  \end{align}

  Fix any $\epsilon > 0$. Then by \eqref{eq:ulln_continuity} and
  \eqref{eq:ulln_g_continuity}, for each $x \in B$ there exists a
  $\delta(x,\epsilon) > 0$ such that $\Xmet(\tilde{x}, x) < \delta(x,\epsilon)$
  implies
  $|g(z,\tilde{x})-g(z,x)| < \frac{d(z) + \Exp d(Z)}{2\Exp
    d(Z)}\epsilon$.
  Let
  $B^{\delta(x, \epsilon)}(x) = \left\{ \tilde{x} \in X: \Xmet(\tilde{x}, x) <
    \delta(x, \epsilon) \right\}$
  and note that $\cup_{x \in B} B^{\delta(x, \epsilon)}(x) \supset B$. Then since $B$ is
  compact, there exists a finite subcovering
  $\left\{ B^{\delta_i}(x_i) \right\}_{i=1}^m$, $\delta_i := \delta(x_i, \epsilon)$ such that
  \begin{align*}
    \bigcup_{i=1}^{m} B^{\delta_i}(x_i) \supset B.
  \end{align*}
  Let $x \in B$ and let $i$ be an index such that $x \in B^{\delta_i}(x_i)$. Then
  \begin{align*}
      \left| \frac{1}{N} \sum_{j=1}^N g(Z_j,x) \right| 
      &\le \left| \frac{1}{N} \sum_{j=1}^N (g(Z_j,x) - g(Z_j,x_i)) \right| 
        + \left| \frac{1}{N} \sum_{j=1}^N g(Z_j,x_i) \right| \\
      &\le \frac{\epsilon}{2\Exp d(Z)} 
        \left| \frac{1}{N} \sum_{j=1}^N d(Z_j) + \Exp d(Z) \right| 
         + \left| \frac{1}{N} \sum_{j=1}^N g(Z_j,x_i) \right|. \\
  \end{align*}
  Taking the supremum over $x$ and using the subcovering yields
  \begin{align*}
      \sup_{x \in B} \left| \frac{1}{N} \sum_{j=1}^N g(Z_j,x) \right| 
      &\le \max_{i=1,\dots,m} \sup_{x \in B^{\delta_i}(x_i)} 
            \left[ \frac{\epsilon}{2\Exp d(Z)} 
           \left| \frac{1}{N} \sum_{j=1}^N d(Z_j) + \Exp d(Z) \right| 
      	      + \left| \frac{1}{N} \sum_{j=1}^N g(Z_j,x_i) \right| \right] \\
      &\le \frac{\epsilon}{2\Exp d(Z)} \left| \frac{1}{N} \sum_{j=1}^N d(Z_j) 
        + \Exp d(Z) \right| 
        + \max_{i=1,\dots,m} \left| \frac{1}{N} \sum_{j=1}^N g(Z_j,x_i) \right|. \\
  \end{align*}
  Then the Strong Law of Large Numbers gives
    \begin{align*}
      \limsup_{N \to \infty} \left[\sup_{x \in B} 
           \left| \frac{1}{N} \sum_{j=1}^N g(Z_j,x) \right| 
      \right] 
      &\le \limsup_{N \to \infty} \left(
      	   \frac{\epsilon}{2\Exp d(Z)} \left| \frac{1}{N} \sum_{j=1}^N d(Z_j) 
      	+ \Exp d(Z) \right| 
        + \max_{i=1,\dots,m} \left| \frac{1}{N} \sum_{j=1}^N g(Z_j,x_i) \right| 
	\right)\\
      &\le \epsilon \frac{2\Exp d(Z)}{2\Exp d(Z)} 
        + \max_{i=1,\dots,m} \Exp g(Z,x_i) = \epsilon \quad a.s.
    \end{align*}
  where the last equality follows from the fact that $\Exp g(Z,x)=0$ for all $x$. Thus, we have
  \begin{align*}
    \Omega_\epsilon \coloneqq 
    \left\{ \lim_{N \to \infty} 
    \sup_{x \in B} \left| \frac{1}{N} \sum_{j=1}^N g(Z_j,x) \right| < \epsilon \right\}
  \end{align*}
  has probability $1$ for all $\epsilon > 0$. Then taking
  $\Omega_{0} = \cap_{k=1}^{\infty} \Omega_{\frac{1}{k}}$ and invoking
  the continuity of measures,
  \begin{align*}
    \Prob \left\{ \Omega_{0} \right\} 
    = \Prob \left\{ \lim_{N \to \infty} \sup_{x \in B} 
           \left| \frac{1}{N} \sum_{j=1}^N g(Z_j,x) \right| = 0 
           \right\} 
    = \lim_{K \to \infty} \Prob \left\{ \cap_{k=1}^{K} \Omega_{\frac{1}{k}} \right\} 
    = 1.
  \end{align*}
  which is the desired result.
\end{proof}

We now use this uniform law of large numbers to show that for large
$N$, the growth in the log-likelihood (normalized by $\frac{1}{N}$) 
for a vector field $\vfield$ can be written in terms of the
observation error and the difference between the scalar fields
associated with $\vfield$ and $\vtrue$.

\begin{Prop}
  \label{thm:l2_decompose_unif}
   Let $\left\{(t_j,\x_j)\right\}_{j=1}^\infty$ be a sequence of
  observation points independently and identically uniformly distributed in $[0,T] \times
  \spatdom$.
  Fix a $\vtrue \in \vfspaceb$ and draw associated data points
  $\left\{ \data_j \right\}_{j=1}^\infty$ according to 
  \begin{align}
    \data_j = \G_j\left(\vtrue \right) + \noise_j
    	\label{eq:LLN:data:def}
  \end{align}
  for i.i.d.~observational noise
  $\noise_j \sim N\left(0,\sigma_\noise^2\right)$ independent of
  $\left\{ \data_j \right\}_{j=1}^\infty$ and the
  parameter-to-observable (forward) maps $\G_j$ given by
  \eqref{eq:for:m:consist}.  Then, for any $r>0$,
  \begin{align}
    \limsup_{N \to \infty} \sup_{\vfield \in B_\vfspaceb^r\left( \vtrue \right)} 
    \Biggl| &\frac{1}{N} \sum_{j=1}^{N} \left( \data_{j} - \G_j(\vfield) \right)^2 
      - \left( \sigma_\noise^2 
      + \frac{1}{2T}\norm{ \pairSol(\vtrue) - \pairSol(\vfield) }_{\pairSolSp}^2 \right) 
     \Biggr| = 0,
    \label{eq:l2_decompose_unif}
  \end{align}
  almost surely,  where $\pairSol$ is the paired solution operator as in
  \cref{def:S_pair}. 
\end{Prop}
\begin{proof}
    Referring to \eqref{eq:LLN:data:def} and expanding we have
    \begin{align}
       \frac{1}{N} \sum_{j=1}^{N} \left( \data_{j} - \G_{j}(\vfield) \right)^2
      &= \frac{1}{N} \sum_{j=1}^{N} \noise_{j}^2 
      +  \frac{2}{N} \sum_{j=1}^{N} \noise_{j}\left( \G_{j}(\vtrue) - \G_{j}(\vfield) \right) 
      +  \frac{1}{N} \sum_{j=1}^{N} \left( \G_{j}(\vtrue) - \G_{j}(\vfield) \right)^2 
      \notag\\
      &:= \frac{1}{N}\sum_{j =1}^{N}(T_{1,j} + 2 T_{2,j} + T_{3,j}),
        \label{eq:uLLN:app:exp}
    \end{align}
    for any $N \geq 1$.
  We will now focus on each of the three terms on the right hand side. 

  \textbf{Terms involving }$\mathbf{T_{1,j}}$: For this first term, the law of
  large numbers yields
  \begin{align}
     \lim_{N \to \infty}\frac{1}{N}\sum_{j =1}^NT_{1,j} 
       = \Exp \noise_j^2 = \sigma_\noise^2 \quad a.s.
    \label{eq:l2_decompose_term1}
  \end{align}
  Also, these terms are independent of, and therefore uniform in,
  $\vfield \in H$.

  \textbf{Terms involving }$\mathbf{T_{2,j}}$: Here we establish 
  uniform convergence using \cref{thm:ulln}.    
  Denote $z = (z_\noise, z_t, z_{x_1}, z_{x_2}) \in \R^4$ and define
  \begin{align}
    f_i(z,\vfield)
      = z_\noise
      \left(\pdesol(z_t,(z_{x_1}, z_{x_2}),\vtrue,\pdesol_0^{(i)}) 
        - \pdesol(z_t,(z_{x_1}, z_{x_2}),\vfield,\pdesol_0^{(i)} ) \right),
    \label{eq:T2:f:uLLN}
  \end{align}
  for $i = 1,2$.
  Let us verify the conditions required by \cref{thm:ulln} for $f_i$.    
  Note that by our assumption on $\noise$,
  $\Exp \noise^2 = \sigma_\noise^2 <
  \infty$.  Thus, by the maximum principle, 
  \begin{align}
    \Exp f_i((\eta, t, \x),\vfield)^2 <
    \norm{ \pdesol_0^{(i)} }_{L^\infty}^2\Exp |\noise|^2 <
    \infty, \quad i = 1,2
    \label{eq:T2:f:uLLN:c1}
  \end{align}
  which corresponds to \eqref{eq:ulln_moment}.
  Moreover by the continuity identified in \cref{thm:pt_obs_continuous}, for all
  $\epsilon > 0$, $\vfield \in H$,
  there exists a $\delta = \delta(\vfield)$ such that
  \begin{align}
    \norm{\vfield-\tilde{\vfield}} < \delta(\vfield) 
    \implies |f(z,\vfield) - f(z,\tilde{\vfield})| 
       \le |z_\noise|\epsilon, \text{ for }i = 1,2,
    \label{eq:T2:f:uLLN:c2}
  \end{align}
  thus verifying \eqref{eq:ulln_continuity}.
  Finally observe that since $\left\{ \noise_j \right\}$ and
  $\left\{ t_j, \x_j \right\}$ are independent, so are
  $\{ \noise_j \}$ and $\{ \G_{j}(\vtrue) - \G_{j}(\vfield)\}$. 
  Furthermore, using $\Exp \noise_j = 0$, we have for any $j \geq 1$,
  $\Exp \noise_{j}  \left( \G_{j}(\vtrue) - \G_{j}(\vfield) \right)=$  
     $\Exp \noise_{j} \Exp \left( \G_{j}(\vtrue) - \G_{j}(\vfield) \right) 
      = 0$,
  thus
  \begin{align}
    \E f_i(Z, \vfield) = 0, \quad \text{for $i = 1,2$}.
    \label{eq:T2:f:uLLN:c3}
  \end{align}

  Fix any $r > 0$.  Since $B_\vfspaceb^r\left( \vtrue \right)$ is
  compact in $\vfspacea$ by the Rellich-Kondrachov Theorem
  (\cite[Corollary A.5]{robinson2001infinite}), and using
  \eqref{eq:T2:f:uLLN:c1}--\eqref{eq:T2:f:uLLN:c3}, \cref{thm:ulln}
  yields
  \begin{align}
    \limsup_{N \to \infty}\sup_{\vfield \in B_\vfspaceb^r\left( \vtrue \right)} 
      \left| \frac{2}{N}\sum_{j =1}^N T_{2,j}(\vfield)\right| 
    \notag
      &\leq
      \limsup_{N \to \infty}\sup_{\vfield \in B_\vfspaceb^r\left( \vtrue \right)} 
      \left| \frac{2}{N} \sum_{l = 0}^{\lceil N/2 \rceil -1} T_{2,2l+1}(\vfield)\right| \\ 
       &\quad +       
    \limsup_{N \to \infty}\sup_{\vfield \in B_\vfspaceb^r\left( \vtrue \right)} 
    \left| \frac{2}{N} \sum_{l = 1}^{\lfloor N/2 \rfloor} T_{2,2l}(\vfield)\right|  
    = 0.
    \label{eq:l2_decompose_term2}
  \end{align}

  \textbf{Terms involving }$\mathbf{T_{3,j}}$: 
  Here we begin by observing that, since the observations $(t_j,\x_j)$
  are uniformly distributed on $[0,T] \times \spatdom$,
  \begin{align*}
    \Exp T_{3,j} :=
    \begin{cases}
    \frac{1}{T} 
    \norm{ \pdesol(\cdot, \vtrue, \pdesol_0^{(1)}) 
        - \pdesol(\cdot, \vfield,  \pdesol_0^{(1)}) }_{L^2\left( [0,T]\times\spatdom \right)}^2
    & \text{ if $j$ is even,}\\
    \frac{1}{T} 
    \norm{ \pdesol(\cdot, \vtrue, \pdesol_0^{(2)}) 
         - \pdesol(\cdot, \vfield, \pdesol_0^{(2)}) }_{L^2\left( [0,T]\times\spatdom \right)}^2
    & \text{ if $j$ is odd.}
    \end{cases}
  \end{align*}  
  To show the uniform convergence of these terms, denote
  $z=( z_t, z_{x_1}, z_{x_2}) \in \R^3$ and define
  \begin{align*}
    f_i(z,\vfield)= \left(\pdesol(z_t, z_{x_1}, z_{x_2},\vtrue, \pdesol_0^{(i)}) 
         - \pdesol(z_t, z_{x_1}, z_{x_2},\vfield, \pdesol_0^{(i)}) \right)^2,
  \end{align*}
  for $i = 1,2$.  Invoking the the maximum principle as in
  \eqref{eq:T2:f:uLLN:c1} we have, 
  \begin{align}
    \Exp f_i(Z,\vfield)^2 < 16\norm{ \pdesol_0^{(i)} }_{L^\infty}^4 <
    \infty, \quad i =1,2,
    \label{eq:T3:f:uLLN:c1}
  \end{align}
  where here $Z$ is distributed uniformly as
  $(t_j,\x_j)$.  Also, by \cref{thm:pt_obs_continuous}, for all
  $\epsilon > 0$ there exists a $\delta$ such that
  \begin{equation}
    \norm{\vfield-\tilde{\vfield}} < \delta(\vfield) 
    \implies |f(z,\vfield) - f(z,\tilde{\vfield})|  \le \epsilon.
    \label{eq:T3:f:uLLN:c2}
  \end{equation}
  Note that in this case the bound is independent of $z$.
  Noting once again that $B_\vfspaceb^r\left( \vtrue \right)$ is a compact 
  subset of $\vfspacea$ and that \eqref{eq:T3:f:uLLN:c1}, \eqref{eq:T3:f:uLLN:c2}
  yield the conditions 
  \eqref{eq:ulln_moment}, \eqref{eq:ulln_continuity} we find
  with \cref{thm:ulln} that
  \begin{align}
    \limsup_{N \to \infty} \sup_{\vfield \in B_\vfspaceb^r\left( \vtrue \right)} 
    &\left| \frac{1}{N} \sum_{j=1}^{N} \left( \G_j(\vtrue) - \G_j(\vfield) \right)^2 
      - \frac{1}{2T}\norm{ \pairSol(\vtrue) - \pairSol(\vfield) }_{\pairSolSp}^2  
     \right|
    \notag\\
     \leq
      &\limsup_{N \to \infty}\sup_{\vfield \in B_\vfspaceb^r\left( \vtrue \right)} 
      \left| \frac{1}{N} 
       \!\!\!\!\!\sum_{l = 0}^{\;\;\lceil N/2 \rceil -1} \!\!\!\!\!\!\! T_{3,2l+1}(\vfield)
        - \frac{\Exp f_1(Z, \vfield)}{2} \right|  
        \notag\\
       &+       
    \limsup_{N \to \infty}\sup_{\vfield \in B_\vfspaceb^r\left( \vtrue \right)} 
    \left| \frac{1}{N} \!\!\sum_{l = 1}^{\lfloor N/2 \rfloor} \!\!T_{3,2l}(\vfield)
       - \frac{\Exp f_2(Z, \vfield)}{2} \right|  
    =\,0.
        \label{eq:l2_decompose_term3}
  \end{align}

  Referring back to \eqref{eq:uLLN:app:exp} and assembling the three estimates
  \eqref{eq:l2_decompose_term1}, \eqref{eq:l2_decompose_term2}, and
  \eqref{eq:l2_decompose_term3}, we arrive at \eqref{eq:l2_decompose_unif}.
  The proof is complete.
\end{proof}

\section{Identification of the Scalar Field}
\label{sec:C:true:scalar}

In this section, we show that the Bayesian posterior measure for $N$
point observations $\mu_{\data}^N$ converges to background flows that
closely match the true scalar field $\pdesol(\vtrue)$.  The idea is to
use the decomposition of the log-likelihood given in
\cref{thm:l2_decompose_unif} along with the assumptions
\eqref{eq:strng:spt:cond}, \eqref{eq:assum:prior_tail_2} to gain
control of tail events.

The main result is as follows
\begin{Prop}[Identification of true $\pdesol$]
  \label{thm:consistency}
  Take $\left\{ (t_j,\x_j) \right\}$, $\vtrue$, $\G_j$, and
  $\left\{ \data_j \right\}$ to be the observation points, the `true
  vector field', the forward map, and the data, respectively that
  are defined as in and satisfy the conditions of
  \cref{thm:post:conv}. Let $\mu_\data^N$ be the associated the
  posterior measures for $N$ observations given by
  \eqref{eq:rand:post:measure}, where we assume that the conditions
  \eqref{eq:strng:spt:cond}, \eqref{eq:assum:prior_tail_2} are 
  enforced.  Then, for any $\delta > 0$,
  \begin{align}
    \mu_\data^N \left(\Vdel\right) \to 1, \quad \text{as $N \to \infty$},
    \label{eq:dirac}
  \end{align}
  on a set of full measure, 
  where, cf. \eqref{eq:forward:th:ball},
  \begin{align}
    \Vdel = \left\{ \vfield \in \vfspacea 
      : \norm{ \pairSol(\vtrue) - \pairSol(\vfield) }_{\pairSolSp} < \delta \right\}.
    \label{eq:close:pf:con}
  \end{align}
\end{Prop}
\begin{Rmk}
Note that \cref{thm:Sinv_continuous2} with $\vfield_0 = \vtrue$ also has implications for $\Vdel$. Indeed, these two characterizations will be combined in \cref{sec:cons:conclusion} to prove the main result.
\end{Rmk}

Before turning directly to the proof of \cref{thm:consistency} we
first establish a lemma that derives some simple but useful consequences
of the assumptions \eqref{eq:strng:spt:cond},
\eqref{eq:assum:prior_tail_2}.  We recycle this lemma
again for later use in \cref{sec:cons:conclusion}.
\begin{Lem}
  \label{thm:prior_radius}
  Suppose that $\vtrue \in V$, and that $\mu_0$ satisfies
  \eqref{eq:strng:spt:cond}.  Define the measures $\mu^N_\data$ as in
  (\ref{eq:rand:post:measure}) and assume that the condition
  \eqref{eq:assum:prior_tail_2} is maintained.\footnote{See also
    \eqref{eq:assum:prior_tail_1} in \cref{rmk:prior:const}.} Then, on
  a set $\tilde{\Omega}$ of full measure, for any
  $\delta,\epsilon > 0$, there exists an
  $R = R(\delta, \epsilon, \omega) > 0$ (but independent of $N$) so that both
  \begin{equation}
    \mu_0\left( \Vdel \cap B_\vfspaceb^{R}\left( \vtrue \right) \right)>0 
    \quad \text{ and } \quad
    \mu_\data^N\left( \left( B_\vfspaceb^R\left( \vtrue \right) \right)^c \right) < \epsilon,
    \label{eq:prior_radius}
  \end{equation}
  for every $N \geq 1$.
\end{Lem}

\begin{proof}
  Let $\delta > 0$ and $\epsilon > 0$.  By
  \cref{thm:S_continuous_pair}, there exists an $r>0$ such that
  $B_\vfspaceb^{r}\left( \vtrue \right) \subset \Vdel$.  Thus for any
  $R > r$, we observe that
  \begin{align*}
    \mu_0\left( \Vdel \cap B_\vfspaceb^{R}\left( \vtrue \right) \right) 
    \ge \mu_0\left( \Vdel \cap B_\vfspaceb^{r}\left( \vtrue \right) \right) 
    = \mu_0\left( B_\vfspaceb^{r}\left( \vtrue \right) \right) 
    > 0
  \end{align*}
  by \eqref{eq:strng:spt:cond}.

  To establish the other condition in \eqref{eq:prior_radius}, use \cref{rmk:prior:eq:cond} to choose 
  $\tilde{R} > r $ such that
  \begin{align*}
      \mu_\data^N\left( \left( B_\vfspaceb^{\tilde{R}}\left(\mathbf{0}\right)
        \right)^c \right) < \epsilon.
  \end{align*}
  Then selecting $R = \tilde{R} + \norm{\vtrue}_\vfspaceb$ ensures that $R > r$,
  maintaining the first condition in \eqref{eq:prior_radius}, and 
  further guaranteeing that
  $( B_\vfspaceb^{R}\left(\vtrue\right) )^c \subset
  ( B_\vfspaceb^{\tilde{R}}\left(\mathbf{0}\right) )^c$,
  and thus
  $\mu_\data^N\left( \left( B_\vfspaceb^{R}\left(\vtrue\right)
    \right)^c \right) < \epsilon$,
  as desired for the second condition in \eqref{eq:prior_radius}. The 
  proof is now complete.
\end{proof}

With \cref{thm:prior_radius} in hand we now turn to the proof of the
main result of this section.

\begin{proof}[Proof of \cref{thm:consistency}]
  We begin by specifying an event on which
  \eqref{eq:dirac} will be established.  Take
  \begin{align*}
    \tilde{\Omega} = \bigcap_{n =1}^\infty 
    \left\{ \omega \in \Omega: 
    \eqref{eq:l2_decompose_unif} \text{ holds for } r = n \right\}.
  \end{align*}
  According to \cref{thm:l2_decompose_unif} this is a set of full measure.
  Fix any $\omega \in \tilde{\Omega}$.  All of the constants and
  statements that follow will implicitly depend on this sample $\omega$. 

  Take arbitrary $\delta, \epsilon > 0$.  As in
  \cref{thm:prior_radius}, select $R > 0$ so that both
  \begin{align*}
    \mu_0\left( \Vdelt \cap B_\vfspaceb^{R}\left(\vtrue\right) \right)>0
    \quad \text{ and } \quad
    \mu_\data^N\left( \left( B_\vfspaceb^R\left(\vtrue\right) \right)^c \right) 
    < \frac{\epsilon}{2}.
  \end{align*}
  For these values of $R$ and $\delta$, we invoke
  \cref{thm:l2_decompose_unif} and infer that there exists an $N_1>0$
  such that for all $N \geq N_1$ and every
  $\vfield \in B_\vfspaceb^R\left(\vtrue\right)$,
  \begin{align*}
    \left| \frac{1}{N} \sum_{j=1}^{N} 
       \left( 
         \data_{j} - \G_j\left(\vfield\right) \right)^2 
       - \left( \sigma_\noise^2 
         + \frac{1}{2T}\norm{ \pairSol(\vtrue) - \pairSol(\vfield) }_{\pairSolSp}^2 
       \right) \right| < \frac{\delta^2}{8T}.
  \end{align*}
  Then for every
  $\vfield \in \Vdelt \cap B_\vfspaceb^R\left(\vtrue\right)$ and
  $N \geq N_1$, we have
  \begin{align*}
   \frac{1}{N} \sum_{j=1}^{N} \left( \data_{j} - \G_j\left(\vfield\right) \right)^2 
    &< \sigma_\noise^2 
      + \frac{1}{2T}\norm{ \pairSol(\vtrue) - \pairSol(\vfield) }_{\pairSolSp}^2 
      + \frac{\delta^2}{8T} 
      < \sigma_\noise^2 + \frac{\delta^2}{4T}.
  \end{align*}
  Similarly, for every
  $\vfield \in \Vdel^c \cap B_\vfspaceb^R\left(\vtrue\right)$ and
  $N \geq N_1$, we have
  \begin{align*}
    \frac{1}{N} \sum_{j=1}^{N} 
    \left( \data_{j} - \G_j\left(\vfield\right) \right)^2 
    &> \sigma_\noise^2 
    + \frac{1}{2T}\norm{ \pairSol(\vtrue) - \pairSol(\vfield) }_{\pairSolSp}^2 
    - \frac{\delta^2}{8T}
    \ge \sigma_\noise^2 + \frac{3\delta^2}{8T}.
  \end{align*}
  Now, leveraging
  $\mu_0\left(\Vdelt \cap B_\vfspaceb^R\left(\vtrue\right) \right)>0$,
  we choose $N_2$ such that
  \begin{align*}
     \frac{ 1 }{ \mu_0\left(\Vdelt\cap B_\vfspaceb^R\left(\vtrue\right)\right) } 
     \exp \left[ -\frac{\delta^2}{16T\sigma_\noise^2}N_2 \right] < \frac{\epsilon}{2}.
  \end{align*}
  Then, for all $N \geq \max\{N_1, N_2\}$, we have (cf. \eqref{eq:rand:post:measure})
  \begin{align*}
      \mu_\data^N(\Vdel^c\cap B_\vfspaceb^R\left(\vtrue\right)) 
      &\le \frac{ {\displaystyle \int_{\Vdel^c\cap B_\vfspaceb^R\left(\vtrue\right)} }
        \exp \left[ -\frac{1}{2\sigma_\noise^2} 
            \sum\limits_{j=1}^N \left(\data_j-\G_j\left(\vfield\right) \right)^2\right] 
       \mu_0(d\vfield) }
        {{\displaystyle \int_{\Vdelt\cap B_\vfspaceb^R\left(\vtrue\right)}} \exp 
        \left[ -\frac{1}{2\sigma_\noise^2} 
               \sum\limits_{j=1}^N \left(\data_j-\G_j\left(\vfield\right) \right)^2\right] 
        \mu_0(d\vfield) } \\
      &< \frac{ {\displaystyle \int_{\Vdel^c\cap B_\vfspaceb^R\left(\vtrue\right)}} 
          \exp \left[ -\frac{1}{2\sigma_\noise^2} 
            N \left( \sigma_\noise^2 
                   + \frac{3\delta^2}{8T} \right) \right] \mu_0(d\vfield) }
         {{\displaystyle \int_{\Vdelt\cap B_\vfspaceb^R\left(\vtrue\right)} }
            \exp \left[ -\frac{1}{2\sigma_\noise^2} 
           N\left( \sigma_\noise^2 + \frac{\delta^2}{4T} \right)\right] 
          \mu_0(d\vfield) } \\
      &= \exp \left[ -\frac{N\delta^2}{16T\sigma_\noise^2} \right] 
        \frac{ \mu_0\left(\Vdel^c\cap B_\vfspaceb^R\left(\vtrue\right)\right) }
        { \mu_0\left(\Vdelt\cap B_\vfspaceb^R\left(\vtrue\right)\right) } 
      \le \frac{\exp \left[ -\frac{N\delta^2}{16T\sigma_\noise^2} \right] 
          }{ \mu_0\left(\Vdelt\cap B_\vfspaceb^R\left(\vtrue\right)\right) } 
      < \frac{\epsilon}{2}.
  \end{align*}
  Then 
  \begin{align*}
    \mu_\data^N\left(\Vdel^c\right) 
    \le \mu_\data^N\left(\Vdel^c\cap B_\vfspaceb^R\left(\vtrue\right)\right) 
       + \mu_\data^N\left( \left( B_\vfspaceb^R\left(\vtrue\right) \right)^c \right) 
       < \frac{\epsilon}{2} + \frac{\epsilon}{2} = \epsilon.
  \end{align*}
  Thus, since $\epsilon$ and $\omega \in \tilde{\Omega}$ are arbitrary
  we conclude that for any $\delta > 0$,
  $\mu_\data^N \left(\Vdel^c\right) \to 0$ as $N \to \infty$ a.s. The desired
  result \eqref{eq:dirac} follows, completing the proof of 
  \cref{thm:consistency}.
\end{proof}

\section{Convergence of Posterior Measures to the True Vector Field}
\label{sec:cons:conclusion}

We now combine the continuity of the inverse map
(\cref{thm:Sinv_continuous}) and the convergence of the posterior
measure to $\pdesol(\vtrue)$ 
(\cref{thm:consistency}) to finally prove our main result
\cref{thm:post:conv}, i.e., to show that as the number of
observations goes to infinity, the posterior converges weakly to a
Dirac measure centered at $\vtrue$.

\begin{proof}[Proof of \cref{thm:post:conv}]
  Let $\epsilon > 0$. Let $A$ be an open subset of $\vfspacea$. To
  show weak convergence, according to Portmanteau's Theorem (see e.g. \cite[Section 2]{billingsley2013convergence}) we need to show
  \begin{align*}
    \liminf_{N \to \infty} \mu_\data^N(A) \ge \delta_{\vtrue}(A).
  \end{align*}
  If $\vtrue \not\in A$, then $\delta_{\vtrue}(A) = 0$ so the result
  is trivial in this case.

  Now consider $\vtrue \in A$ and fix any sample $\omega$ on the set
  $\tilde{\Omega}$ of full measure for which (\ref{eq:dirac}) in
  \cref{thm:consistency} holds. Fix any $\epsilon >0$.  As guaranteed by
  \cref{thm:prior_radius}, we can choose $R>0$ so that
  \begin{align*}
    \mu_\data^N\left( \left( B_\vfspaceb^R(\vtrue) \right)^c \right) < \frac{\epsilon}{2}.
  \end{align*}
  Since $A$ is open there exists an $\gamma > 0$ such that
  $B_\vfspacea^{\gamma}(\vtrue) \subset A$. Then, by
  \cref{thm:Sinv_continuous2}, there exists a $\delta > 0$ such that
  $\Vdel \cap B_\vfspaceb^R(\vtrue) \subset
  B_\vfspacea^{\gamma}(\vtrue) \subset A$. As such
  \begin{align*}
    \mu_\data^N(A) 
    \ge \mu_\data^N\left( B_\vfspacea^\gamma(\vtrue) \right) 
    \ge \mu_\data^N\left( \Vdel \cap B_\vfspaceb^R(\vtrue) \right) 
    &\ge \mu_\data^N\left( \Vdel \right) 
         - \mu_\data^N\left( \left( B_\vfspaceb^R(\vtrue) \right)^c \right) 
    \ge \mu_\data^N\left( \Vdel \right) - \frac{\epsilon}{2}.
  \end{align*}
  However, \cref{thm:consistency} ensures that there exists an
  $N^\star$ such that for all $N > N^\star$,
  \begin{align*}
    \mu_\data^N\left( \Vdel \right) > 1 - \frac{\epsilon}{2}.
  \end{align*}
  Therefore for all $N > N^\star$,
  \begin{align*}
    \mu_\data^N(A) 
    \ge \mu_\data^N\left( \Vdel \right) - \frac{\epsilon}{2} 
    > 1 - \epsilon = \delta_{\vtrue}(A) - \epsilon.
  \end{align*}
  Since $\epsilon > 0$ and $\omega$ were arbitrary to begin with,
  $\liminf_{N \to \infty} \mu_\data^N(A) \ge \delta_{\vtrue}(A)$ with
  probability $1$ completing the proof of \cref{thm:post:conv}.
\end{proof}

\appendix 

\section{Energy Estimates for Continuity of the Solution Map}\label{sec:energyest}

In this appendix we provide some of the {\em a priori} estimates leading to 
\cref{def:adr_weak}.  As noted above, a suitable Galerkin approximation
of (\ref{eq:adr}) can be implemented to provide rigorous justification
for the forthcoming formal manipulations.

Let us begin with the $L^2$-based estimates.  Since $\vfield$ is divergence
free, we have that
$\frac{d}{dt} \| \pdesol \|^2 + 2\kappa \| \nabla \pdesol\|^2 = 0$ so
that for any $T > 0$,
$\pdesol \in L^2([0,T]; H^1(\spatdom)) \cap L^\infty([0,T];
L^2(\spatdom)).$ Turning to the estimate for $\partial_t \pdesol$ we have
\begin{align}
  \| \partial_t \pdesol\|_{H^{-1}} \leq \kappa \|\pdesol\|_{H^1} + \|\vfield \cdot \nabla \pdesol\|_{H^{-1}}.
  \label{eq:der:est:H_1_1}
\end{align}
Regarding the second term on the right hand, using that $\vfield$ is 
divergence free and H\"older's inequality
\begin{align}
   \|\vfield \cdot \nabla \pdesol\|_{H^{-1}}
   = \sup_{\|\phi\|_{H^1} =1} \left| \int \vfield \cdot \nabla \pdesol \phi d\x \right|
   \leq C\| \vfield\|_{L^p} \|\pdesol\|_{L^q}
    \label{eq:der:est:H_1_2}
\end{align}
where $p^{-1} + q^{-1} = 2^{-1}$. Let us now recall the Sobolev
embedding in spatial dimension $d = 2$ which entails the bound
 \begin{align}
   \| f \|_{L^p} \leq C \| f \|_{H^{r}}
   \quad \text{ for any } r \geq 1 - \frac{2}{p}, \text{ with } 2 \leq p < \infty,
   \label{eq:sob:embedding}
\end{align}
for any sufficiently smooth $f$ and where the constant $C$ depends
only on the size of the periodic box, $p$, and $r$.  Thus, with our
assumption that $\vfield \in H^s$ for some $s > 0$ it now follows from
\eqref{eq:der:est:H_1_1}, \eqref{eq:der:est:H_1_2}, and
\eqref{eq:sob:embedding} that
$\partial_t \pdesol \in L^2([0,T]; H^{-1})$.

Regarding the claimed continuity in $L^2$ we consider any
$\pdesol^{(1)}, \pdesol^{(2)}$ solving (\ref{eq:adr}) and
corresponding to divergence free $\vfield^{(1)}, \vfield^{(2)}$.
Taking $\psi = \pdesol^{(2)} - \pdesol^{(1)}$ and
$\bfU = \vfield^{(2)} - \vfield^{(1)}$ we have
\begin{align}
 \partial_t \psi
      = \kappa \Delta \psi
        - \bfU \cdot \nabla \pdesol^{(2)} - \vfield^{(1)} \cdot \nabla \psi.
 \label{eq:diff:AD:eq} 
\end{align}
Multiplying \eqref{eq:diff:AD:eq} by $\psi$, integrating and using that
are both $\vfield^{(1)}, \bfU$ are divergence free, we obtain
\begin{align}
  \frac{1}{2}\frac{d}{dt} \| \psi \|^2 + \kappa \| \nabla \psi\|^2
   =  \int \bfU \cdot \nabla \psi \pdesol^{(2)} d\x.
  \label{eq:s:en:diff}
\end{align}
Now, H\"older's inequality yields
\begin{align*}
  \left|\int \bfU \cdot \nabla \psi \pdesol^{(2)} d\x  \right|
  \leq \|\bfU\|_{L^p} \| \nabla \psi\| \|\pdesol^{(2)}\|_{L^q}
\end{align*}
which holds for any $2 \leq p, q \leq \infty$ such that
$p^{-1} + q^{-1} = 2^{-1}$.  Observe that, by choosing
$2 < q < \infty$ sufficiently large, obtain a $p = \frac{2 q}{q -2}$
such that, according to (\ref{eq:sob:embedding})
$\|\bfU\|_{L^p} \leq C \|\bfU\|_{H^s}$ where $s > 0$ is the given
degree of regularity for $\vfield^{(1)}, \vfield^{(2)}$.  With this
observation, another application of (\ref{eq:sob:embedding}), this time
for $\|\pdesol^{(2)}\|_{L^q}$, and Young's inequality we have
\begin{align*}
 \left|\int \bfU \cdot \nabla \psi \pdesol^{(2)} d\x  \right|
  \leq C\|\bfU\|_{H^s}  \| \psi \|_{H^1} \|\pdesol^{(2)}\|_{H^1}
  \leq \frac{\kappa}{2} \| \psi \|_{H^1}^2 + C\|\bfU\|_{H^s}^2  \|\pdesol^{(2)}\|_{H^1}^2.
\end{align*}
Combining this bound with (\ref{eq:s:en:diff}) we find, for any $T > 0$,
\begin{align*}
  \sup_{t \in [0,T]} \|\pdesol^{(1)}(t) - \pdesol^{(2)}(t)\|^2 \leq 
      \|  \pdesol^{(1)}_0 - \pdesol^{(2)}_0 \|^2 + 
  C\|\bfU\|_{H^s}^2  \int_0^T\|\pdesol^{(2)}(t')\|_{H^1}^2dt'
\end{align*}
from which the desired continuity in the $L^2$ case now follows.

Before proceeding to the higher order, $s > 0$, estimates, let us introduce
further notations and recall some fundamental inequalities.  For any
$r \geq 0$, we take $\Lambda^r := (- \Delta)^{r/2}$ acting on elements
in $H^r(\spatdom)$.  In other words
\begin{align*}
  \Lambda^r f = \sum_{\kbf \in \Z^2} \left( 2\pi \right)^{r} \| \kbf \|^{r} c_{\kbf} e^{2 \pi i \kbf \cdot \x }
  \quad \text{ for any } f = \sum_{\kbf \in \Z^2} c_{\kbf} e^{2 \pi i \kbf \cdot \x }
\end{align*}
and we have $\| \Lambda^r f \| = \|f\|_{H^r}$.  We have the following useful
interpolation inequality 
\begin{align}
  \| \Lambda^r f \| \leq 
  \| \Lambda^{\gamma_l} f\|^{\frac{\gamma_u - r}{\gamma_u - \gamma_l}}
  \| \Lambda^{\gamma_u} f\|^{\frac{r - \gamma_l}{\gamma_u - \gamma_l}}
  \label{eq:Intrp:Ineq}
\end{align}
valid for any $0 \leq \gamma_l < r < \gamma_u$; see
e.g.~\cite{robinson2001infinite}.  We will also make use of the
fractional Leibniz inequality or Kato-Ponce inequality:
\begin{align}
  \| \Lambda^{r} (fg) \|_{L^m} \leq 
  C ( \| \Lambda^{r} f\|_{L^{p_1}} \| g \|_{L^{q_1}} +
  \| f\|_{L^{p_2}} \| \Lambda^{r}  g \|_{L^{q_2}} 
  )
  \label{eq:Kato:Ponce}
\end{align}
which is valid for any $r \geq 0$, $1 < m < \infty$ and
$1 < p_i, q_i \leq \infty$ with $m^{-1} = p^{-1}_j + q^{-1}_j$ for
$j = 1,2$ and where the constant $C$ is independent of any suitably
smooth $f, g$.  See \cite{grafakos2014kato, muscalu2013classical} for
further details.

With these preliminaries in hand, now suppose $\pdesol$ solves
(\ref{eq:adr}). Applying the operator $\Lambda^s$ to (\ref{eq:adr}),
multiplying by $\Lambda^s \pdesol$ and integrating over $\mathbb{T}^2$
we obtain
\begin{align}
  \frac{1}{2}\frac{d}{dt} \| \Lambda^s \pdesol\|^2 
  + \kappa \| \Lambda^{s+1} \pdesol \|^2 
  = \int \Lambda^s(\vfield\cdot \nabla \pdesol) \Lambda^s \pdesol d\x.
       \label{eq:high:order:1}
\end{align}
With H\"older's inequality and \eqref{eq:Kato:Ponce} we find
\begin{align}
  \left| \int \Lambda^s(\vfield\cdot \nabla \pdesol) 
          \Lambda^s \pdesol d\x\right|
  \leq C \|\Lambda^s \pdesol \|_{L^p} 
      (\|\Lambda^s\vfield\| \|\Lambda^1 \pdesol\|_{L^q} + 
       \|\vfield\|_{L^q} \|\Lambda^{s+1} \pdesol\|),
  \label{eq:Gen:interp:1}
\end{align}
which holds for any $1 < p, q < \infty$ such
that  $1-\frac{1}{p} = \frac{1}{2} + \frac{1}{q}$.  Noting
that $q = \frac{2p}{2p -2 - p} \to 2$ as $p \to \infty$, using
the Sobolev embedding (\ref{eq:sob:embedding}) 
and then the interpolation inequality \eqref{eq:Intrp:Ineq}, 
we thus find, for some $0 < s' < s \wedge 1$,
\begin{align}
  \left| \int \Lambda^s(\vfield\cdot \nabla \pdesol) 
          \Lambda^s \pdesol d\x\right|
  \leq& C \|\Lambda^{1+s'} \pdesol \| \| \Lambda^{1+ s} \pdesol\| \|\Lambda^s\vfield\| \\
  \leq& C \|\Lambda^{1+s} \pdesol \|^{2-(s- s')} \|\Lambda^s \pdesol \|^{s - s'} 
  \|\Lambda^s\vfield\|
  \notag\\
 \leq& \kappa \|\Lambda^{1+s} \pdesol \|^{2} + 
       C\|\Lambda^s \pdesol \|^{2} 
  \|\Lambda^s\vfield\|^{\frac{2}{(s-s')}}.
       \label{eq:high:order:2}
\end{align}
Combining \eqref{eq:high:order:1}, \eqref{eq:high:order:2}
and rearranging we obtain
\begin{align*}
 \frac{d}{dt} \| \Lambda^s \pdesol\|^2 
  + \kappa \| \Lambda^{s+1} \pdesol \|^2 
  \leq        C\|\Lambda^s \pdesol \|^{2} 
     \|\Lambda^s\vfield\|^{\frac{2}{(s-s')}}.
\end{align*}
This bound and Gr\"onwall's inequality reveals 
\begin{align*}
  \sup_{t \in [0,T]} \|\pdesol(t)\|^2_{H^s} 
  \leq  \exp( T C \|\Lambda^s\vfield\|^{\frac{2}{(s-s')}}) \|\pdesol_0\|_{H^s}^2.
\end{align*}
Using this bound and integrating in time yields
\begin{align*}
  \int_0^T \| \Lambda^{s+1} \pdesol \|^2  
  \leq 2 \exp( T  \|\Lambda^s\vfield\|^{\frac{2}{(s-s')}}) \|\pdesol_0\|_{H^s}^2,
\end{align*}
which indeed shows that for any $T > 0$,
$\pdesol \in L^2([0,T];H^{s+1}(\spatdom))\cap L^\infty([0,T];$ $
H^s(\spatdom))$.
We turn next to the estimates for $\partial_t \pdesol$. 
Here analogous to (\ref{eq:der:est:H_1_1}) we just need
a suitable estimate for $\|\vfield \cdot \nabla \pdesol\|_{H^{s-1}}$.
For any $s > 0$ this amounts to 
\begin{align*}
   \|\vfield \cdot \nabla \pdesol\|_{H^{s-1}}
   := \sup_{\|\phi\|_{H^{s+1}} =1} \left| \int \vfield \cdot \nabla \pdesol \phi d\x \right|
   \leq C\| \vfield\|_{L^p} \|\pdesol\|_{L^q} \sup_{\|\phi\|_{H^{s+1}} =1} \|\nabla \phi\|_{L^r}
\end{align*}
for any $1 \leq p,q,r \leq \infty$ with $p^{-1} + q^{-1} + r^{-1} = 1$ and 
where again we have used that $\vfield$ is divergence free.  By picking 
$p = r > 2$ such that $L^p \subset H^s$ according to (\ref{eq:sob:embedding})
we finally obtain   
\begin{align*}
 \|\vfield \cdot \nabla \pdesol\|_{H^{s-1}}
  \leq C \| \vfield\|_{H^s} \|\pdesol\|_{H^{s+1}}
\end{align*}
and thus conclude that $\partial_t \pdesol \in L^2([0,T]; H^{s-1})$
for any $T > 0$.

We finally address the claimed continuity of the data to solution map
in $H^s$.  Adopting the same notations as in (\ref{eq:diff:AD:eq}), we
have
\begin{align}
  \frac{1}{2} \frac{d}{dt} \| \Lambda^s \psi\|^2
       + \kappa \| \Lambda^{s+1} \psi\|^2      
     = -\int 
        \Lambda^{s}(\bfU \cdot \nabla \pdesol^{(2)} - \vfield^{(1)} \cdot \nabla \psi)
        \Lambda^{s} \psi d\x
    := T_1 + T_2.
       \label{eq:high:order:3}
\end{align}
Regarding $T_1$, we estimate as in \eqref{eq:Gen:interp:1}, \eqref{eq:high:order:2} and find,
\begin{align}
  |T_1|
  \leq& C \|\Lambda^s \psi \|_{L^p} 
      (\|\Lambda^s \bfU\| \|\Lambda^1 \pdesol^{(2)}\|_{L^q} + 
       \|\bfU\|_{L^q} \|\Lambda^{s+1} \pdesol^{(2)}\|)
  \notag\\
  \leq& C \|\Lambda^{1+s} \psi \| \|\Lambda^s\bfU\| \|\Lambda^{s+1} \pdesol^{(2)}\|
  \leq \kappa \|\Lambda^{1+s} \psi \|^2
        + C \|\Lambda^s\bfU\|^2 \|\Lambda^{s+1} \pdesol^{(2)}\|^2
  \label{eq:Gen:interp:2}
\end{align}
For $T_2$ we proceed in precisely the same fashion as (\ref{eq:high:order:2})
and find 
\begin{align}
 |T_2| \leq \kappa \|\Lambda^{1+s} \psi \|^{2} + 
       C\|\Lambda^s \psi \|^{2} 
  \|\Lambda^s\vfield^{(1)}\|^{\frac{2}{(s-s')}}
       \label{eq:high:order:4}
\end{align}
Combining the identity \eqref{eq:high:order:3} with the estimates
\eqref{eq:Gen:interp:2}, \eqref{eq:high:order:4} and rearranging
appropriately we obtain
\begin{align*}
  \frac{d}{dt} \| \Lambda^s \psi\|^2
  \leq C\|\Lambda^s \psi \|^{2} \|\Lambda^s\vfield^{(1)}\|^{\frac{2}{(s-s')}}
      + C \|\Lambda^s\bfU\|^2 \|\Lambda^{s+1} \pdesol^{(2)}\|^2,
\end{align*}
and hence, with Gr\"onwall's inequality,
\begin{align*}
  \| \Lambda^s \psi(t)\|^2 &\leq 
     \exp(C \|\Lambda^s\vfield^{(1)}\|^{\frac{2}{(s-s')}}t) \| \psi(0)\|^2 \\
      &\quad + C \|\Lambda^s\bfU\|^2 
  \int_0^t \exp(C \|\Lambda^s\vfield^{(1)}\|^{\frac{2}{(s-s')}}(t - t')) \|\Lambda^{s+1} \pdesol^{(2)}(t')\|^2dt'.
\end{align*}
Thus given the already established {\em a priori} bounds on $\pdesol^{(2)}$ in $H^s$
and our standing assumption concerning the regularity of $\vfield^{(1)}$ we 
have
\begin{align*}
  \sup_{t \in [0,T]}\| \pdesol^{(1)}(t) - \pdesol^{(2)}(t)\|^2_{H^s}
  \leq C (\| \pdesol^{(1)}_0 - \pdesol^{(2)}_0\|^2_{H^s} 
            + \| \vfield^{(1)} - \vfield^{(2)}\|^2_{H^s}),
\end{align*}
from which the desired continuity in $H^s$ now follows.

\section*{Acknowledgments}
This work was supported in part by the National Science Foundation
under grants DMS-1313272 (NEGH), DMS-1816551 (NEGH), DMS-1522616 (JTB), and DMS-1819110 (JTB); the National Institute for
Occupational Safety and Health under grant 200-2014-59669 (JTB); and the
Simons Foundation under grant 515990 (NEGH). We would like to thank G. Didier for drawing our attention to the question of Bayesian consistency for the advection-diffusion problem. We would also like to express our appreciation to J. Foldes, S. McKinley, A. Stuart, and J. Whitehead for additional useful discussions and references.

\addcontentsline{toc}{section}{References}
\begin{footnotesize}
\bibliographystyle{plain}
\bibliography{references}
\end{footnotesize}

\vspace{.5in}
\begin{multicols}{2}
\noindent
Jeff Borggaard\\
{\footnotesize Department of Mathematics\\
Virginia Tech\\
Web: \url{https://www.math.vt.edu/people/jborggaa/}\\
Email: \href{mailto:jborggaard@vt.edu}{\nolinkurl{jborggaard@vt.edu}}} \\[.5cm]
Nathan Glatt-Holtz\\ {\footnotesize
Department of Mathematics\\
Tulane University\\
Web: \url{http://www.math.tulane.edu/~negh/}\\
Email: \href{mailto:negh@tulane.edu}{\nolinkurl{negh@tulane.edu}}} \\[.2cm]

\columnbreak

 \noindent Justin Krometis\\
{\footnotesize
Advanced Research Computing\\
Virginia Tech\\
Web: \url{https://www.arc.vt.edu/justin-krometis/}\\
Email: \href{mailto:jkrometis@vt.edu}{\nolinkurl{jkrometis@vt.edu}}} \\[.2cm]
 \end{multicols}

\end{document}